\documentclass[12pt, a4paper, reqno]{amsart}
\usepackage{amsmath}
\usepackage{amsfonts}
\usepackage{amssymb}
\usepackage{color}
\usepackage{comment,graphicx}
\usepackage[T1]{fontenc}
\usepackage{url}
\usepackage{ae}
\usepackage{mathtools}
\usepackage{esint}

\def\phi{\varphi}
\def\rho{\varrho}
\def\epsilon{\varepsilon}
\numberwithin{equation}{section}
\theoremstyle{plain}
\newtheorem{theorem}[equation]{Theorem}
\newtheorem{lemma}[equation]{Lemma}

\newtheorem{corollary}[equation]{Corollary}
\theoremstyle{definition}
\newtheorem{definition}[equation]{Definition}

\theoremstyle{remark}
\newtheorem{remark}[equation]{Remark}

\renewcommand{\leq}{\leqslant}
\renewcommand{\geq}{\geqslant}

\begin{document}
\title[Multiplication of Besov and Triebel-Lizorkin spaces]{\textbf{Mixed
multiplication of Besov and Triebel-Lizorkin spaces}}
\author[D. Drihem]{Douadi Drihem}
\address{M'sila University, Department of Mathematics, Laboratory of
Functional Analysis and Geometry of Spaces, M'sila 28000, Algeria.}
\email{douadidr@yahoo.fr, douadi.drihem@univ-msila.dz}
\thanks{ }
\date{\today }

\begin{abstract}
This paper is concerned with proving some embeddings of the form%
\begin{equation*}
F_{p_{1},q}^{s_{1}}\cdot B_{p_{2},\infty }^{s_{2}}\cdot ...\cdot
B_{p_{m},\infty }^{s_{m}}\hookrightarrow F_{p,q}^{s_{1}},\quad m\geq 2.
\end{equation*}%
The different embeddings obtained here are under certain restrictions on the
parameters. In particular, we improve some results of pointwise
multiplication on Triebel-Lizorkin spaces. Franke-Jawerth\ embeddings, the $%
\pm $ - method of Gustaffson-Peetre and the relation between Hardy spaces
and Triebel-Lizorkin spaces are the main tools.
\end{abstract}

\maketitle

\section{Introduction}

Approximately 45 years ago Peetre \cite{P}\ and Triebel \cite{T1}, \cite{T2}%
, independent from each other, have applied a special decomposition of the
product $f\cdot g$ to investigate the product 
\begin{equation*}
A_{p_{1},q_{1}}^{s_{1}}\cdot A_{p_{2},q_{2}}^{s_{2}}\hookrightarrow
A_{p,q}^{s}
\end{equation*}%
in case of $p_{1}=p_{2}=p$. Here $A_{p,q}^{s}$ stands for either the Besov
space $B_{p,q}^{s}$ or the Triebel-Lizorkin space $F_{p,q}^{s}$. This
method, nowadays called paramultiplication, was applied later on by many
authors. Concerning earlier contributions to this subject the paper of
Yamazaki \cite{Y} deals with the situation where $p_{1}=p_{2}=p$, whereas
Sickel treats the cases with $p_{1}=p_{2}\neq p$. Hanouzet \cite{H85} has
investigated $p_{1}\neq p_{2}$ but restricted to Besov spaces. In the a
recent work of Sickel and Triebel the case $p_{1}\neq p_{2}$ is also studied
and a rather complete set of necessary conditions is given. A new necessary
and sufficient conditions are given in the paper of J. Johnsen \cite{J}, see
also J. Marschall \cite{M1} and \cite{M2}, , see Amann \cite{A} for Sobolev
and Besov spaces.

Concerning the case $F_{p_{1},q_{1}}^{s_{1}}\cdot
B_{p_{2},q_{2}}^{s_{2}}\hookrightarrow F_{p,q}^{s}$. This case were studied
by Franke \cite{F}\ and Marschall in \cite{M1} and \cite{M2} with $p_{1}=p$.
In recent paper \cite{DM2}, the case $p_{1}\neq p\ $is also studied where
known sufficient conditions for pointwise multiplication have been improved.
More close to this contribution is the book of T. Runst, W. Sickel \cite{RS}%
, see also V. Maz'ya, T. Shaponiskova \cite{MSh}.

The motivation to study the problem of multiplication on function spaces
comes from applications to partial differential equations, see for example 
\cite{Wu}, where estimates of the product on function spaces are handy\ in
dealing with the quadratic nonlinear term in many partial differential
equations, see also, H. Bahouri, J.-Y. Chemin, and R. Danchin \cite{BCD11},
V.G. Maz'ya and T.O. Shaposhnikova \cite{MSh}, and Zeidler \cite{Z}.
Furthermore, estimates of products of functions have played a key role in
investigations of composition operators, see \cite[Chapter 5]{RS}.

Recently, Meyries and Veraar \cite{MM12},\ and Lindemulder \cite{L21}
considered Besov and Bessel potential spaces $B_{p,q}^{s}(\mathbb{R}%
^{n},w),H_{p}^{s}(\mathbb{R}^{n},w)$ with respect to the weight $%
w(x,t)=|t|^{\alpha },x\in \mathbb{R}^{n-1},t\in \mathbb{R}$. Under some
suitable assumptions on $s,p,q$ and $\alpha $, they observed that the
characteristic function of the half space is a pointwise multiplier for $%
B_{p,q}^{s}(\mathbb{R}^{n},w),H_{p}^{s}(\mathbb{R}^{n},w)$.

The purpose of this paper is to study the $m$-linear map%
\begin{equation*}
F_{p_{1},q}^{s_{1}}\cdot A_{p_{2},\infty }^{s_{2}}\cdot ...\cdot
A_{p_{m},\infty }^{s_{m}}\hookrightarrow F_{p,q}^{s_{1}},\quad m\geq 2,
\end{equation*}%
induced by 
\begin{equation*}
(f_{1},f_{2},...,f_{m})\longrightarrow f_{1}\cdot f_{2}\cdot ...\cdot f_{m}.
\end{equation*}

We want to present here, briefly, the contents of our work. In Section 2 we
recall the definition of the different spaces and some necessary tools. We
shall apply the method of paramultilication to decompose the product%
\begin{equation*}
f_{1}\cdot f_{2}\cdot ...\cdot f_{m}.
\end{equation*}%
Products in spaces with positive smoothness are given in subsection 3.1. In
Subection 3.2 the product in space with negative smoothness is given.

We will adopt the following convention throughout this paper. As usual, we
denote by $\mathbb{R}^{n}$ the $n$-dimensional real Euclidean space, $%
\mathbb{N}$ the collection of all natural numbers and $\mathbb{N}_{0}=%
\mathbb{N}\cup \{0\}$. For a multi-index $\alpha =(\alpha _{1},...,\alpha
_{n})\in \mathbb{N}_{0}^{n}$, we write $\left\vert \alpha \right\vert
=\alpha _{1}+...+\alpha _{n}$. The Euclidean scalar product of $%
x=(x_{1},...,x_{n})$ and $y=(y_{1},...,y_{n})$ is given by $x\cdot
y=x_{1}y_{1}+...+x_{n}y_{n}$. If $a\in \mathbb{R}$, then we put $a_{+}=\max
(0,a)$.

\noindent\ As usual $L_{p}(\mathbb{R}^{n})$ for $0<p\leq \infty $ stands for
the Lebesgue spaces on $\mathbb{R}^{n}$ for which 
\begin{equation*}
\big\|f\mid L_{p}(\mathbb{R}^{n})\big\|=\big\|f\big\|_{p}=\Big(\int_{\mathbb{%
R}^{n}}\left\vert f(x)\right\vert ^{p}dx\Big)^{1/p}<\infty ,\text{\quad }%
0<p<\infty
\end{equation*}%
and 
\begin{equation*}
\big\|f\mid L_{\infty }(\mathbb{R}^{n})\big\|=\big\|f\big\|_{\infty }=%
\underset{x\in \mathbb{R}^{n}}{\text{ess-sup}}\left\vert f(x)\right\vert
<\infty .
\end{equation*}

By $\mathcal{S}(\mathbb{R}^{n})$ we denote the Schwartz space of all
complex-valued, infinitely differentiable and rapidly decreasing functions
on $\mathbb{R}^{n}$ and by $\mathcal{S}^{\prime }(\mathbb{R}^{n})$ the dual
space of all tempered distributions on $\mathbb{R}^{n}$. We define the
Fourier transform of a function $f\in \mathcal{S}(\mathbb{R}^{n})$ by%
\begin{equation*}
\mathcal{F}(f)(\xi )=\overset{\wedge }{f}(\xi )=(2\pi )^{-n/2}\int_{\mathbb{R%
}^{n}}e^{-ix\cdot \xi }f(x)dx,\quad \xi \in \mathbb{R}^{n}.
\end{equation*}%
Its inverse is denoted by $\mathcal{F}^{-1}f$. Both $\mathcal{F}$ and $%
\mathcal{F}^{-1}$ are extended to the dual Schwartz space $\mathcal{S}%
^{\prime }\left( \mathbb{R}^{n}\right) $ in the usual way.

If $s\in \mathbb{R}\ $and $0<q\leq \infty $, then $\ell _{q}^{s}$ is the set
of all sequences $\{f_{j}\}_{j\in \mathbb{N}_{0}}$ of complex numbers such
that%
\begin{equation*}
\big\|\left\{ f_{j}\right\} _{j\in \mathbb{N}_{0}}\mid \ell _{q}^{s}\big\|=%
\Big(\sum_{j=0}^{\infty }2^{jsq}\left\vert f_{j}\right\vert ^{q}\Big)%
^{1/q}<\infty
\end{equation*}%
with the obvious modification if $q=\infty $. If $s=0$ then we shortly
denote $\ell _{q}^{0}$ by $\ell _{q}$.

If $0<p,q\leq \infty $, then the space $L_{p}(\ell _{q})$ (resp. $\ell
_{q}(L_{p})$) is the set of the sequences\ of functions $\{f_{j}\}_{j\in 
\mathbb{N}_{0}}$ such that%
\begin{equation*}
\big\|\{f_{j}\}_{j\in \mathbb{N}_{0}}\mid L_{p}(\ell _{q})\big\|=\Big\|\Big(%
\sum_{j=0}^{\infty }\left\vert f_{j}\right\vert ^{q}\Big)^{1/q}\Big\|%
_{p}<\infty ,\quad 0<p<\infty
\end{equation*}%
\begin{equation*}
\Big(\text{resp. }\big\|\{f_{j}\}_{j\in \mathbb{N}_{0}}\mid \ell _{q}(L_{p})%
\big\|=\Big(\sum_{j=0}^{\infty }\big\|f_{j}\big\|_{p}^{q}\Big)^{1/q}<\infty %
\Big),
\end{equation*}%
with the obvious modification if $q=\infty .$

Given two quasi-Banach spaces $X$ and $Y$, we write $X\hookrightarrow Y$ if $%
X\subset Y$ and the natural embedding of $X$ in $Y$ is continuous. We shall
use $c$ to denote positive constant which may differ at each appearance.

\section{Functions spaces}

In this section we present the Fourier analytical definition of Besov and
Triebel-Lizorkin spaces and recall their basic properties. Also, we present
some useful results we need along the paper. We begin by a specific
resolution of unity. Let $\Psi $\ be a function\ in $\mathcal{S}(\mathbb{R}%
^{n})$\ satisfying $\Psi (x)=1$\ for\ $\left\vert x\right\vert \leq 1$\ and\ 
$\Psi (x)=0$\ for\ $\left\vert x\right\vert \geq \tfrac{3}{2}$.\ We put $%
\varphi _{0}(x)=\Psi (x)$, $\varphi _{1}(x)=\Psi (x/2)-\Psi (x)$\ and 
\begin{equation*}
\varphi _{j}(x)=\varphi _{1}(2^{-j+1}x)\text{\quad for\quad }j=2,3,....
\end{equation*}%
Then we have supp $\varphi _{j}\subset \left\{ x\in \mathbb{R}%
^{n}:2^{j-1}\leq \left\vert x\right\vert \leq 3\cdot 2^{j-1}\right\} $, $%
\varphi _{j}\left( x\right) =1$\ for $3\cdot 2^{j-2}\leq \left\vert
x\right\vert \leq 2^{j}$\ and $\Psi (x)+\sum_{j=1}^{\infty }\varphi
_{j}(x)=1 $ for all $x\in \mathbb{R}^{n}$.\ The system of functions $\left\{
\varphi _{j}\right\} _{j\in \mathbb{N}_{0}}$\ is called a smooth dyadic
resolution of unity. We define the convolution operators $\Delta _{j}$\ by
the following: 
\begin{equation*}
\Delta _{j}f=\mathcal{F}^{-1}\varphi _{j}\ast f,\quad j\in \mathbb{N}\quad 
\text{and}\quad \Delta _{0}f=\mathcal{F}^{-1}\Psi \ast f,\quad f\in \mathcal{%
S}^{^{\prime }}(\mathbb{R}^{n}).
\end{equation*}%
Thus we obtain the Littlewood-Paley decomposition $f=\sum_{j=0}^{\infty
}\Delta _{j}f$ of all $f\in \mathcal{S}^{^{\prime }}(\mathbb{R}^{n})$ $($%
convergence in $\mathcal{S}^{^{\prime }}(\mathbb{R}^{n}))$.

We define the convolution operators $Q_{j}$, $j\in \mathbb{N}_{0}$ by the
following: 
\begin{equation*}
Q_{j}f=\mathcal{F}^{-1}\Psi _{j}\ast f,\quad j\in \mathbb{N}_{0},
\end{equation*}%
where $\Psi _{j}=\Psi (2^{-j}\cdot ),j\in \mathbb{N}_{0}$ and we see that 
\begin{equation*}
Q_{j}f=\sum_{k=0}^{j}\Delta _{k}f
\end{equation*}%
for any $j\in \mathbb{N}_{0}.$\vskip5pt

Now we present the definition and a summary of basic results for Besov and
Triebel-Lizorkin spaces. In the interest of brevity, we shall only develop
those aspects which are relevant for us in the sequel. More detailed
accounts can be found in, e.g., J. Peetre \cite{P}, T. Runst and W. Sickel, 
\cite{RS} and H. Triebel [16, 17, 18].

\begin{definition}
\label{def-Be-Tr-Li spaces}(i)\ Let\ $s\in \mathbb{R}$\ and $0<p,q\leq
\infty $. The Besov space\ $B_{p,q}^{s}$ is the collection of all $f\in 
\mathcal{S}^{^{\prime }}(\mathbb{R}^{n})$ such that 
\begin{equation*}
\big\|f\mid B_{p,q}^{s}\big\|=\big\|\{\Delta _{j}f\}_{j\in \mathbb{N}%
_{0}}\mid \ell _{q}^{s}\left( L_{p}\right) \big\|<\infty .
\end{equation*}%
(ii)\ Let $s\in \mathbb{R},0<p<\infty $\ and $0<q\leq \infty $. The
Triebel-Lizorkin space $F_{p,q}^{s}$ is the collection of all $f\in \mathcal{%
S}^{^{\prime }}(\mathbb{R}^{n})$\ such that%
\begin{equation*}
\big\|f\mid F_{p,q}^{s}\big\|=\big\|\{\Delta _{j}f\}_{j\in \mathbb{N}%
_{0}}\mid L_{p}\left( \ell _{q}^{s}\right) \big\|<\infty .
\end{equation*}
\end{definition}

Now we give some properties of $F_{p,q}^{s}$ and $B_{p,q}^{s}$ which are of
interest for us, see T. Runst and W. Sickel, \cite{RS} and H. Triebel [16,
17, 18].

\begin{lemma}
\label{embedding}$\mathrm{(i)}$\textit{\ The spaces }$B_{p,q}^{s}$\textit{\
and }$F_{p,q}^{s}$\textit{\ are quasi Banach spaces }$($\textit{Banach space
in the case}$\ p,q\geq 1)$ \textit{and in any case}%
\begin{equation*}
\mathcal{S}(\mathbb{R}^{n})\hookrightarrow B_{p,q}^{s}\hookrightarrow 
\mathcal{S}^{\prime }(\mathbb{R}^{n})\quad \text{and}\quad \mathcal{S}(%
\mathbb{R}^{n})\hookrightarrow F_{p,q}^{s}\hookrightarrow \mathcal{S}%
^{\prime }(\mathbb{R}^{n}).
\end{equation*}%
$\mathrm{(ii)}$\textit{\ Let }$s_{i}\in \mathbb{R},\ 0<p_{i}<\infty $\textit{%
\ }$($\textit{resp. }$0<p_{i}\leq \infty )$\textit{\ and }$0<q_{i}\leq
\infty $\textit{\ }$($\textit{with }$i=0,1).$\textit{\ If }$s_{0}>s_{1}$\ 
\textit{and}$\mathit{\ }p_{0}=p_{1}$,\textit{\ or }$s_{0}\geq s_{1}\ $%
\textit{and}$\ s_{0}-\tfrac{n}{p_{0}}=s_{1}-\tfrac{n}{p_{1}}$\textit{,}$%
\mathit{\ }(q_{0}\leq q_{1}$ \textit{for Besov space}$)$,\textit{\ then it
holds}%
\begin{equation*}
F_{p_{0},\,q_{0}}^{s_{0}}\hookrightarrow F_{p_{1},\,q_{1}}^{s_{1}}\quad (%
\text{\textit{resp. }}B_{p_{0},\,q_{0}}^{s_{0}}\hookrightarrow
B_{p_{1},\,q_{1}}^{s_{1}}).
\end{equation*}%
$\mathrm{(iii)}$\textit{\ Let }$s,s_{i}\in \mathbb{R},\ 0<p,p_{i}<\infty $%
\textit{\ and }$0<q,q_{i}\leq \infty $\textit{\ }$($\textit{with }$i=0,1),$%
\textit{\ such that }$s_{0}-\frac{n}{p_{0}}=s-\frac{n}{p}=s_{1}-\frac{n}{%
p_{1}}.$\textit{\ If }$s_{0}>s>s_{1}\ $\textit{and}$\ q_{0}\leq p\leq q_{1}$,%
\textit{\ or }$s_{0}=s=s_{1}$,$\ q_{0}\leq \min \left( p,q\right) \ $\textit{%
and}$\ q_{1}\geq \max \left( p,q\right) $,\textit{\ then it holds }%
\begin{equation*}
B_{p_{0},\,q_{0}}^{s_{0}}\hookrightarrow F_{p,\,q}^{s}\hookrightarrow
B_{p_{1},\,q_{1}}^{s_{1}},\quad \text{Franke-Jawerth\ embeddings.}
\end{equation*}
\end{lemma}

As usual, by $\left\langle A_{0},A_{1}\right\rangle _{\theta }$ we denote
the result of the $\pm $-method of Gustaffson-Peetre applied to quasi-Banach
spaces $A_{0}$ and $A_{1}$. The next theorem was proved in \cite{FJ}.

\begin{theorem}
\label{interpolation}\textit{Let }$0<\theta <1,0<p_{0},p_{1}<\infty
,0<q_{0},q_{1}\leq \infty $\textit{\ and }$s_{0},s_{1}\in \mathbb{R}$\textit{%
. Then}%
\begin{equation*}
\left\langle
F_{p_{0},\,q_{0}}^{s_{0}},F_{p_{1},\,q_{1}}^{s_{1}}\right\rangle _{\theta
}=F_{p,\,q}^{s}
\end{equation*}%
\textit{provided that }$s=\left( 1-\theta \right) s_{0}+\theta s_{1}$\textit{%
, }$\frac{1}{p}=\frac{1-\theta }{p_{0}}+\frac{\theta }{p_{1}}$\textit{\ and }%
$\frac{1}{q}=\frac{1-\theta }{q_{0}}+\frac{\theta }{q_{1}}.$
\end{theorem}

\begin{remark}
\label{interpolation1}The $\pm $-method has the so called interpolation
property$\mathrm{.}$ Denote by $\mathcal{L}(A,B)$ the set of all bounded
linear operators from $A$ to $B$. Then this means that if $T\in \mathcal{L}%
(A_{i},B_{i})$, $i=1,2$, we have $T$ maps $\left\langle
A_{0},B_{0}\right\rangle _{\theta }$ into $\left\langle
A_{1},B_{1}\right\rangle _{\theta }$ and 
\begin{equation*}
\big\|T\mid \left\langle A_{0},B_{0}\right\rangle _{\theta }\longrightarrow
\left\langle A_{1},B_{1}\right\rangle _{\theta }\big\|\leq \big\|T\mid
A_{0}\longrightarrow A_{1}\big\|^{1-\theta }\big\|T\mid A_{0}\longrightarrow
A_{1}\big\|^{\theta },
\end{equation*}%
which plays a crucial role in this paper.
\end{remark}

Now we recall some results which are useful for us. The next lemma is a
Hardy-type inequality which is easy to prove.

\begin{lemma}
\label{Hardying}\textit{Let }$0<\gamma <1$\textit{\ and }$0<q\leq \infty $%
\textit{. Let }$\left\{ \varepsilon _{k}\right\} _{k\in \mathbb{N}_{0}}$%
\textit{\ be a sequence of positive real numbers},\textit{\ such that }%
\begin{equation*}
\big\|\{\varepsilon _{k}\}_{k\in \mathbb{N}_{0}}\mid \ell _{q}\big\|%
=A<\infty .
\end{equation*}%
\textit{Then the sequence }$\delta _{k}=\sum_{j=0}^{k}\gamma
^{k-j}\varepsilon _{j}$, $k\in \mathbb{N}_{0}$ \textit{belong to }$\ell _{q}$%
\textit{,\ and the estimate}%
\begin{equation*}
\big\|\{\delta _{k}\}_{k\in \mathbb{N}_{0}}\mid \ell _{q}\big\|\leq c\text{ }%
A
\end{equation*}%
\textit{holds. The constant }$c\ $\textit{depends only on }$\gamma $\textit{%
\ and }$q$\textit{.}
\end{lemma}

\begin{lemma}
\label{ortholemma1}\textit{Let }$\gamma >0$. \textit{For any sequence }$%
\{f_{j}\}_{j\in \mathbb{N}_{0}}$\textit{\ of functions such that }$\mathrm{%
supp}\overset{\wedge }{f_{j}}\subset \{\xi \in \mathbb{R}^{n}:\gamma
^{-1}2^{j}\leq \left\vert \xi \right\vert \leq \gamma 2^{j}\}$\textit{\ we
have}%
\begin{equation*}
\Big\|\sum_{j=0}^{\infty }f_{j}\mid F_{p,q}^{s}\Big\|\leq c\big\|%
\{2^{js}f_{j}\}_{j\in \mathbb{N}_{0}}\mid L_{p}(\ell _{q})\big\|\quad \text{%
if}\quad p<\infty
\end{equation*}%
\textit{and }%
\begin{equation}
\Big\|\sum_{j=0}^{\infty }f_{j}\mid B_{p,q}^{s}\Big\|\leq c\big\|%
\{2^{js}f_{j}\}_{j\in \mathbb{N}_{0}}\mid \ell _{q}(L_{p})\big\|.
\label{estimate1}
\end{equation}%
\textit{The constant }$c$\textit{\ depends on }$s,n,p$\textit{\ and }$\gamma 
$\textit{.}
\end{lemma}

\begin{lemma}
\label{ortholemma2}\textit{Let }$\gamma >1,0<p,q\leq \infty $\textit{\ and }$%
s>n(\frac{1}{p}-1)_{+}$\textit{. For any sequence }$\{f_{j}\}_{j\in \mathbb{N%
}_{0}}$\textit{\ of functions such that }$\mathrm{supp}\overset{\wedge }{%
f_{j}}\subset \{\xi \in \mathbb{R}^{n}:\left\vert \xi \right\vert \leq
\gamma 2^{j}\}$\textit{. Then }$\mathrm{\eqref{estimate1}}$\textit{\ remains
true.}
\end{lemma}

\begin{lemma}
\label{Young'sing}\textit{Let }$0<p\leq q\leq \infty $\textit{\ and }$\gamma
>0$\textit{. Then there exists a constant }$c=c(n,p,q)>0$\textit{\ such that
for all }$f\in L_{p}$\textit{\ with }$\mathrm{supp}\widehat{f}\subset \{\xi
\in \mathbb{R}^{n}:\left\vert \xi \right\vert \leq \gamma \}$\textit{, one
has }%
\begin{equation*}
\big\|f\big\|_{q}\leq c\,\gamma ^{n(\frac{1}{p}-\frac{1}{q})}\big\|f\big\|%
_{p}.
\end{equation*}
\end{lemma}

For Lemma \ref{ortholemma1}, we can see \cite{RS}, while the proof of Lemma %
\ref{ortholemma2} is given in \cite[Lemma 3]{M2}. For the proof of Lemma \ref%
{Young'sing}, see \cite[Section 1.3.2]{T3}.

\begin{lemma}
\label{estimates}\textit{Let }$s\in \mathbb{R}\ $and $0<p<\infty $\textit{.}%
\newline
$\mathrm{(i)}$ \textit{We have}%
\begin{equation*}
\big\|\sup_{j\in \mathbb{N}}\left\vert Q_{j}f\right\vert \big\|_{p}\leq c%
\big\|f\mid F_{p,2}^{0}\big\|,
\end{equation*}%
\textit{for all} $f\in F_{p,2}^{0}$.

\noindent $\mathrm{(ii)}$\textit{\ Let }$j\in \mathbb{N}_{0}$\textit{\ and }$%
f\in B_{p,\infty }^{s}$. Then there exists a positive constant $c$,
independent of $j$, such that%
\begin{equation*}
\big\|Q_{j}f\big\|_{p}\leq c\text{ }\varepsilon _{j}\big\|f\mid B_{p,\infty
}^{s}\big\|,
\end{equation*}%
\textit{where }%
\begin{equation*}
\varepsilon _{j}=\left\{ 
\begin{array}{ccc}
2^{-js}, & \text{if} & s<0, \\ 
1, & \text{if} & s>0, \\ 
(j+1)^{\frac{1}{\min (1,p)}}, & \text{if} & s=0.%
\end{array}%
\right.
\end{equation*}%
$\mathrm{(iii)}$\textit{\ Let }$0<p\leq t\leq \infty ,j\in \mathbb{N}_{0}$%
\textit{\ and }$f\in B_{p,\infty }^{s}$. \textit{Then}%
\begin{equation*}
\big\|\Delta _{j}f\big\|_{t}\leq c\text{ }2^{(\frac{n}{p}-\frac{n}{t}-s)j}%
\big\|f\mid B_{p,\infty }^{s}\big\|,
\end{equation*}%
where the positive constant $c$ is independent of $j$.

\noindent $\mathrm{(iv)}$\textit{\ We have}%
\begin{equation*}
\big\|Q_{j}f\big\|_{t}\leq c\text{ }\varepsilon _{j}\big\|f\mid B_{p,\infty
}^{s}\big\|,
\end{equation*}%
\textit{for all }$j\in \mathbb{N}_{0}$\textit{, all }$f\in B_{p,\infty }^{s}$
\textit{and all }$p<t\leq \frac{1}{(\frac{1}{p}-\frac{s}{n})_{+}}$\textit{,
where }%
\begin{equation*}
\varepsilon _{j}=\left\{ 
\begin{array}{ccc}
1, & \text{if} & p<t<\frac{1}{(\frac{1}{p}-\frac{s}{n})_{+}}, \\ 
(j+1)^{\frac{1}{\min (1,t)}}, & \text{if} & t=\frac{1}{(\frac{1}{p}-\frac{s}{%
n})_{+}}%
\end{array}%
\right.
\end{equation*}%
and the positive constant $c$ is independent of $j$.
\end{lemma}

\begin{proof}
(i) follows from the equality between the local Hardy spaces $h_{p}$ and $%
F_{p,2}^{0}$, (cf. see \cite[ Section 2.2, p. 37, and Theorem 2.5.8/1]{T3}).
For (ii), it is sufficient to see that%
\begin{equation*}
\big\|Q_{j}f\big\|_{p}^{\tau }\leq \sum_{i=0}^{j}2^{-s\tau i}2^{s\tau i}%
\big\|\Delta _{i}f\big\|_{p}^{\tau }\leq c\text{ }\varepsilon _{j}^{\tau }%
\big\|f\mid B_{p,\infty }^{s}\big\|^{\tau },\quad j\in \mathbb{N}_{0},\tau
=\min (1,p)
\end{equation*}%
where if $s<0$ we have used Lemma \ref{Hardying}. Now Lemma \ref{Young'sing}
gives%
\begin{equation*}
\big\|\Delta _{j}f\big\|_{t}\leq c\text{ }2^{(\tfrac{n}{p}-\tfrac{n}{t})j}%
\big\|\Delta _{j}f\big\|_{p}\leq c\text{ }2^{(\tfrac{n}{p}-\tfrac{n}{t}-s)j}%
\big\|f\mid B_{p,\infty }^{s}\big\|
\end{equation*}%
and%
\begin{equation*}
\big\|Q_{j}f\big\|_{t}^{\tau }\leq c\sum_{i=0}^{j}2^{(\tfrac{n}{p}-\tfrac{n}{%
t}-s)\tau i}2^{s\tau i}\big\|\Delta _{i}f\big\|_{p}^{\tau }\leq c\text{ }%
\varepsilon _{j}^{\tau }\big\|f\mid B_{p,\infty }^{s}\big\|^{\tau },\quad
j\in \mathbb{N}_{0},\tau =\min (1,t).
\end{equation*}%
Thus we complete the proof of (iii) and (iv).
\end{proof}

\subsection{Decomposition of the product $\prod\limits_{i=1}^{m}f_{i}$}

For all $f_{i}\in \mathcal{S}^{\prime }(\mathbb{R}^{n})$, $i=1,2,...,m$ the
product $\prod_{i=1}^{m}f_{i}$ is defined by 
\begin{equation*}
\prod_{i=1}^{m}f_{i}=\lim_{j\rightarrow \infty }\prod_{i=1}^{m}Q_{j}f_{i},
\end{equation*}%
if the limit on the right-hand side exists in $\mathcal{S}^{\prime }(\mathbb{%
R}^{n})$. The following decomposition of this product is given in \cite[%
pages 77-78]{S}. We have the following formal decomposition:%
\begin{equation*}
\prod_{i=1}^{m}f_{i}=\sum_{k_{1},...,k_{m}=0}^{\infty
}\prod_{i=1}^{m}(\Delta _{k_{i}}f_{i}).
\end{equation*}%
The fundamental idea is to split $\prod_{i=1}^{m}f_{i}$ into two parts, both
of them being always defined. Let $N$ be a natural number greater than $%
1+\log _{2}3\left( m-1\right) $. Then we have the following decomposition:%
\begin{align*}
\prod_{i=1}^{m}f_{i}=& \sum_{j=0}^{\infty }\left[ Q_{j-N}f_{1}\cdot ...\cdot
Q_{j-N}f_{m-1}\cdot \Delta _{j}f_{m}+...\right. \\
& \left. +\left( \Pi _{l\neq k}Q_{j-N}f_{l}\right) \Delta
_{k}f_{j}+...+\Delta _{j}f_{1}\cdot Q_{j-N}f_{2}\cdot ....\cdot Q_{j-N}f_{m} 
\right] \\
& +\sum_{j=0}^{\infty }\sum^{j}\left( \Delta _{k_{1}}f_{1}\right) \cdot
....\cdot \left( \Delta _{k_{m}}f_{m}\right) ,
\end{align*}%
where the $\sum^{j}$ is taken over all $k\in \mathbb{Z}_{+}^{n}$ such that $%
\max_{\ell =1,...,m}k_{1}=k_{k_{m_{0}}}=j$\quad and\quad $\max_{\ell \neq
m_{0}}\left\vert \ell -k_{\ell }\right\vert <N$. Of course, if $k<0$ we put $%
\Delta _{k}f=0$. Probably $\sum^{j}$ becomes more transparent by restricting
to a typical part, which can be taken to be 
\begin{equation*}
\Big(\prod_{i\in I_{1}}\Delta _{j}f_{i}\Big)\prod_{i\in I_{2}}Q_{_{j}}f_{i},
\end{equation*}%
where 
\begin{equation*}
I_{1},I_{2}\subset \left\{ 1,...,m\right\} ,\quad I_{1}\cap I_{2}=\emptyset
,\quad I_{1}\cup I_{2}=\left\{ 1,...,m\right\} =I,\quad \left\vert
I_{1}\right\vert \geq 2.
\end{equation*}

We introduce the following notations%
\begin{equation*}
\Pi _{1,k}(f_{1},f_{2},...,f_{m})=\sum_{j=N}^{\infty }\Big(\prod_{i\neq
k}Q_{_{j-N}}f_{i}\Big)\Delta _{j}f_{k},\quad k\in I
\end{equation*}%
and%
\begin{equation*}
\Pi _{2}\left( f_{1},f_{2},...,f_{m}\right) =\sum_{j=0}^{\infty }\sum^{j}%
\Big(\prod_{i=1}^{m}\Delta _{k_{i}}f_{i}\Big).
\end{equation*}

The advantage of the above decomposition is based on%
\begin{equation*}
\text{supp }\mathcal{F}\Big(\Big(\prod_{i\neq k}Q_{_{j-N}}f_{i}\Big)\Delta
_{j}f_{k}\Big)\subset \big\{\xi \in \mathbb{R}^{n}:2^{j-1}\leq \left\vert
\xi \right\vert \leq 2^{j+1}\big\},\quad j\geq N
\end{equation*}%
and%
\begin{equation*}
\text{supp }\mathcal{F}\Big(\sum^{j}\Big(\prod_{i=1}^{m}\Delta _{k_{i}}f_{i}%
\Big)\Big)\subset \big\{\xi \in \mathbb{R}^{n}:\left\vert \xi \right\vert
\leq 2^{j+N-2}\big\},\quad j\in \mathbb{N}_{0}.
\end{equation*}%
Finally, if $j\in \mathbb{N}_{0}$, $J=J_{1}\cup J_{2}$ where $J_{1}\subseteq
I_{1}$ and $J_{2}\subseteq I_{2}$, we will use the following notation 
\begin{equation*}
\tilde{Q}_{j}f_{i}=\left\{ 
\begin{array}{ccc}
\Delta _{j}f_{i}, & \text{if} & i\in J_{1}, \\ 
Q_{j}f_{i}, & \text{if} & i\in J_{2}.%
\end{array}%
\right. 
\end{equation*}

\section{Results and their proofs}

We present our results in two different subsections.

\subsection{Products in spaces with positive smoothness}

In this subection we deal with the case 
\begin{equation}
0<s_{1}<s_{2}\leq s_{3}\leq ...\leq s_{m}.  \label{cond1}
\end{equation}

The main result of this subsection is the following theorem.

\begin{theorem}
\label{multi-positive-smothness}\textit{Let }$s_{i}\in \mathbb{R},1\leq
p_{1}<\infty ,0<p_{i}\leq \infty ,i=2,...,m\ $\textit{and }$0<q\leq \infty $.%
\textit{\ Assume }$\mathrm{\eqref{cond1}}$ and $s_{1}<\tfrac{n}{p_{1}}$.%
\textit{\ Suppose further}%
\begin{equation}
\tfrac{1}{p}=\tfrac{1}{p_{1}}+\sum_{i=2}^{m}\tfrac{1}{h_{i}}<1,
\label{cond2.1}
\end{equation}%
where%
\begin{equation}
\big(\tfrac{1}{p_{i}}-\tfrac{s_{i}}{n}\big)_{+}<\tfrac{1}{h_{i}}\leq \tfrac{1%
}{p_{i}},\quad i=2,...,m.  \label{cond2.2}
\end{equation}%
\textit{Then}%
\begin{equation*}
F_{p_{1},q}^{s_{1}}\cdot B_{p_{2},\infty }^{s_{2}}\cdot ...\cdot
B_{p_{m},\infty }^{s_{m}}\hookrightarrow F_{p,q}^{s_{1}},
\end{equation*}%
\textit{holds.}
\end{theorem}

\begin{corollary}
\label{multi-positive-smothness1}Under the hypotheses of Theorem \ref%
{multi-positive-smothness}, then it holds%
\begin{equation*}
F_{p_{1},q}^{s_{1}}\cdot F_{p_{2},\infty }^{s_{2}}\cdot ...\cdot
F_{p_{m},\infty }^{s_{m}}\hookrightarrow F_{p,q}^{s_{1}}.
\end{equation*}
\end{corollary}

The proof of Corollary \ref{multi-positive-smothness1} is immediate because $%
F_{p_{i},\infty }^{s_{i}}\hookrightarrow B_{p_{i},\infty }^{s_{i}}$, $%
i=2,3,...,m$.

\begin{remark}
Corollary \ref{multi-positive-smothness1} is given in \cite[Theorem \ 2.3.3]%
{S} and \cite[ Theorem \ 4.5.2/1]{RS}.
\end{remark}

\begin{remark}
We see that the condition $\sum_{i=1}^{m}s_{i}\geq \max \big(0,\sum_{i=1}^{m}%
\frac{n}{p_{i}}-n\big)$, is necessary for the pointwise multiplication, see 
\cite[ Theorem \ 4.3.1/2]{RS}, which is covered by \eqref{cond2.2}. For
multiplication of type $B_{p_{1},q}^{s_{1}}\cdot B_{p_{2},\infty
}^{s_{2}}\cdot ...\cdot B_{p_{m},\infty }^{s_{m}}\hookrightarrow
B_{p,q}^{s_{1}}$, once again, we refer the reader to the monograph of T.
Runst and W. Sickel \cite[Chapter 4]{RS}.
\end{remark}

\begin{proof}[\textbf{Proof of Theorem \protect\ref{multi-positive-smothness}%
}]
We begin by the estimation of $\Pi _{1,k}\left( f_{1},f_{2},...,f_{m}\right)
.$

\textbf{Estimation of}\textit{\ }$\Pi _{1,k}\left(
f_{1},f_{2},...,f_{m}\right) $. Lemma \ref{ortholemma1} gives%
\begin{equation}
\big\|\Pi _{1,k}(f_{1},f_{2},...,f_{m})\mid F_{p,q}^{s_{1}}\big\|\leq c\Big\|%
\Big\{2^{js_{1}}\Big(\prod_{i\neq k}Q_{_{j-N}}f_{i}\Big)\Delta _{j}f_{k}%
\Big\}_{j\geq N}\mid L_{p}(\ell _{q})\Big\|  \label{estimateQ1}
\end{equation}%
for any $k\in I$. First observe that%
\begin{equation}
\sup_{j\geq N}\big\|Q_{_{j-N}}f_{i}\big\|_{v_{i}}\leq c\big\|f_{i}\mid
F_{v_{i},2}^{0}\big\|,\quad 0<v_{i}<\infty ,i\in I.  \label{estimateQ}
\end{equation}%
Let $\tfrac{1}{p}=\tfrac{1}{p_{1}}+\sum_{i=2}^{m}\tfrac{1}{t_{i}}$ where 
\begin{equation*}
\big(\tfrac{1}{p_{i}}-\tfrac{s_{i}}{n}\big)_{+}<\tfrac{1}{t_{i}}\leq \tfrac{1%
}{p_{i}},\quad i\in I\backslash \{1\}.
\end{equation*}%
We split the estimation into two separate cases.

$\bullet $ \textbf{Case 1. }$k=1$\textbf{.} H\"{o}lder's inequality gives
that the right-hand side of $\mathrm{\eqref{estimateQ1}}$ is dominated by%
\begin{equation*}
c\big\|f\mid F_{p_{1},q}^{s_{1}}\big\|\prod_{i\in I\backslash \{1\}}\big\|%
f_{i}\mid F_{t_{i},2}^{0}\big\|.
\end{equation*}%
Since the $t_{i}$, $i\in I\backslash \{1\}$ may be chosen independent we can
sum up and leads to the restrictions%
\begin{equation*}
\tfrac{1}{p_{1}}+\sum_{i\in I\backslash \{1\}}\big(\tfrac{1}{p_{i}}-\tfrac{%
s_{i}}{n}\big)_{+}<\tfrac{1}{p}\leq \tfrac{1}{p_{1}}+\sum_{i\in I\backslash
\{1\}}\tfrac{1}{p_{i}}\text{.}
\end{equation*}%
Whereas is exactly $\mathrm{\eqref{cond2.1}}$, combined with $\mathrm{%
\eqref{cond2.2}}$. The estimate can be finished by taking into account 
\begin{equation*}
B_{p_{i},\infty }^{s_{i}}\hookrightarrow F_{t_{i},2}^{0},\quad i\in
I\backslash \{1\}.
\end{equation*}

$\bullet $ \textbf{Case 2. }$k\neq 1$\textbf{.} The situation is quite
different and more complicated. We set $\frac{1}{b}=\frac{1}{p_{k}}%
+\sum_{i\in I\backslash \{1,k\}}\tfrac{1}{t_{i}}+\frac{1}{p_{1}}-\frac{s_{1}%
}{n}$, where 
\begin{equation*}
\big(\tfrac{1}{p_{i}}-\tfrac{s_{i}}{n}\big)_{+}<\tfrac{1}{t_{i}}\leq \tfrac{1%
}{p_{i}},\quad i\in I\backslash \{1,k\}.
\end{equation*}%
To prove we additionally do it into the two Subcases 2.1, 2.2 and 2.3.%
\textit{\ }

$\bullet $ \textbf{Subcase 2.1. }$\frac{n}{p_{k}}\geq s_{k}$ or ($s_{1}\leq 
\frac{n}{p_{k}}<s_{k}$). First, assume that $p\leq b$. We put $\frac{1}{p}=%
\frac{1}{p_{k}}+\sum_{i\in I\backslash \{1,k\}}\tfrac{1}{t_{i}}+\frac{1}{%
\tau }$ with $\frac{1}{p_{1}}-\frac{s_{1}}{n}\leq \frac{1}{\tau }<\frac{1}{%
p_{1}}$. Thanks to Lemma \ref{estimates}/(iv) it is obvious that%
\begin{equation*}
\big\Vert Q_{j-N}f_{1}\big\Vert_{\tau }\lesssim \gamma _{k}\big\Vert %
f_{1}\mid F_{p_{1},\infty }^{s_{1}}\big\Vert,\quad j\geq N,
\end{equation*}%
since $\frac{n}{p_{1}}-s_{1}-\frac{n}{\tau }\leq 0$, where 
\begin{equation*}
\gamma _{k}=\left\{ 
\begin{array}{ccc}
j-N+1, & \text{if} & p=b, \\ 
1, & \text{if} & p<b.%
\end{array}%
\right.
\end{equation*}%
H\"{o}lder's inequality gives 
\begin{align*}
& \big\|\Pi _{1,k}(f_{1},f_{2},...,f_{m})\mid B_{p,\min (p,q)}^{s_{1}}\big\|
\\
\leq & c\Big\|\prod_{i\in I\backslash \{k\}}\sup_{j\geq N}\big|%
Q_{_{j-N}}f_{i}\big|\big\{2^{js_{1}}|\Delta _{j}f_{k}|\big\}_{j\geq N}\mid
\ell _{\min (p,q)}(L_{p})\Big\| \\
\leq & c\prod_{i\in I\backslash \{1,k\}}\big\|f_{i}\mid F_{t_{i},2}^{0}\big\|%
\big\|f_{k}\mid B_{p_{k},\infty }^{s_{k}}\big\|\big\Vert f_{1}\mid
F_{p_{1},\infty }^{s_{1}}\big\Vert,
\end{align*}%
because of $s_{1}<s_{k},k\neq 1$. The estimation can be finished by taking
into account 
\begin{equation}
B_{p,\min (p,q)}^{s_{1}}\hookrightarrow F_{p,q}^{s_{1}},\quad
B_{p_{i},\infty }^{s_{i}}\hookrightarrow F_{t_{i},2}^{0},\quad i\in
I\backslash \{1,k\}.  \label{second}
\end{equation}%
Now, assume that $p>b$. Let $u_{1}$ be a positive number satisfying 
\begin{equation*}
\max \Big(0,\tfrac{1}{p}-\tfrac{1}{p_{k}}-\sum_{i\in I\backslash \{1,k\}}%
\tfrac{1}{t_{i}}\Big)<\tfrac{1}{u_{1}}<\tfrac{1}{p_{1}}-\tfrac{s_{1}}{n}.
\end{equation*}%
We put%
\begin{equation*}
\tfrac{1}{v}=\tfrac{1}{p_{k}}+\sum_{i\in I\backslash \{1,k\}}\tfrac{1}{t_{i}}%
+\tfrac{1}{u_{1}},\quad \sigma =s_{1}-\tfrac{n}{p}+\tfrac{n}{v},\quad \beta
=s_{1}-\tfrac{n}{p_{1}}+\tfrac{n}{u_{1}}.
\end{equation*}%
These guarantee the embedding%
\begin{equation}
B_{v,p}^{\sigma }\hookrightarrow F_{p,q}^{s_{1}}.  \label{help3}
\end{equation}%
We need to estimate $\Pi _{1,k}(f_{1},f_{2},...,f_{m})$ in $B_{v,p}^{\sigma
} $\ spaces. H\"{o}lder's inequality yields that%
\begin{align}
& 2^{j\sigma }\big\Vert\prod_{i\in I\backslash \{k\}}Q_{_{j-N}}f_{i}\cdot
\Delta _{j}f_{k}\big\Vert_{v}  \label{new-case} \\
\leq & 2^{j\sigma }\prod_{i\in I\backslash \{1,k\}}\big\Vert Q_{_{j-N}}f_{i}%
\big\Vert_{t_{i}}\big\Vert Q_{_{j-N}}f_{1}\big\Vert_{u_{1}}\big\Vert\Delta
_{j}f_{k}\big\Vert_{p_{k}}.  \notag
\end{align}%
By Lemmas \ref{Young'sing} and \ref{Hardying} we obtain%
\begin{align*}
2^{k\beta }\,\sum_{l=0}^{j-N}\big\Vert\Delta _{l}f_{1}\big\Vert%
_{u_{1}}\lesssim & 2^{k\beta }\,\sum_{l=0}^{j-N}2^{-j\beta }2^{j\beta }%
\big\Vert\Delta _{l}f_{1}\big\Vert_{u_{1}} \\
\lesssim & \big\Vert f_{1}\mid F_{u_{1},\infty }^{\beta }\big\Vert,
\end{align*}%
since $\beta <0$, where the implicit constant is independent of $k$.
Consequently $\mathrm{\eqref{new-case}}$ is dominated by 
\begin{equation*}
c2^{j(\sigma -\beta -s_{k})}\prod_{i\in I\backslash \{1,k\}}\big\|f_{i}\mid
F_{t_{i},2}^{0}\big\|\big\|f_{k}\mid B_{p_{k},\infty }^{s_{k}}\big\|%
\big\Vert f\mid F_{u_{1},\infty }^{\beta }\big\Vert.
\end{equation*}%
Observe that $\sigma -\beta -s_{k}<0$, then the last expression in $\ell
^{p} $-quasi-norm is bounded by%
\begin{equation*}
c\prod_{i\in I\backslash \{1,k\}}\big\|f_{i}\mid B_{p_{i},\infty }^{s_{i}}%
\big\|\big\|f_{k}\mid B_{p_{k},\infty }^{s_{k}}\big\|\big\Vert f\mid
F_{u_{1},\infty }^{\beta }\big\Vert,
\end{equation*}%
where we have used the second embedding of $\mathrm{\eqref{second}}$\textrm{%
. }The desired estimate follows by the embeddings \eqref{help3} and $%
F_{p_{1},\infty }^{s_{1}}\hookrightarrow F_{u_{1},\infty }^{\beta }.$

$\bullet $\textbf{\ Subcase\ 2.2.}\textit{\ }$\frac{n}{p_{k}}<s_{1}<s_{k}$%
\textit{.} We have only the case $p<b$ needs to study. As in Subcase 2.1 we
obtain the desired estimate.

$\bullet $\textbf{\ Subcase\ 2.3. }We consider the case $\tfrac{1}{p}%
=\sum_{i=1}^{m}\tfrac{1}{p_{i}}$. We need to estimate $\Pi
_{1,k}(f_{1},f_{2},...,f_{m})$ in $B_{p,\min (p,q)}^{s_{1}}$\ spaces. First
observe that%
\begin{equation*}
\big\Vert Q_{j-N}f_{i}\big\Vert_{p_{i}}\lesssim \big\Vert f_{i}\mid
F_{p_{i},\infty }^{s_{i}}\big\Vert,\quad j\geq N,i\in I\backslash \{k\},
\end{equation*}%
since $s_{i}>0,i\in I\backslash \{k\}$, see Lemma \ref{estimates}. H\"{o}%
lder's inequality yields that%
\begin{align*}
& 2^{js}\big\Vert\prod_{i\in I\backslash \{k\}}Q_{_{j-N}}f_{i}\cdot \Delta
_{j}f_{k}\big\Vert_{p} \\
\leq & 2^{j(s-s_{k})}\prod_{i\in I\backslash \{k\}}\big\Vert f_{i}\mid
F_{p_{i},\infty }^{s_{i}}\big\Vert2^{js_{k}}\big\Vert\Delta _{j}f_{k}%
\big\Vert_{p_{k}}
\end{align*}%
for any $j\geq N$. Since $s_{1}<s_{k}$, Lemma \ref{ortholemma1} and the
first embeddings of \eqref{second} yield the desired estimate.

\textbf{Estimation of }$\Pi _{2}\left( f_{1},f_{2},...,f_{m}\right) $. The
situation is much more complicated. For simplicity we put 
\begin{equation*}
\tfrac{1}{t}=\sum_{i=1}^{m}\tfrac{1}{p_{i}}.
\end{equation*}%
From $\mathrm{\eqref{cond2.1}}$ and $\mathrm{\eqref{cond2.2}}$, we get 
\begin{equation*}
\tfrac{s_{1}}{n}+\sum_{i=1}^{m}\big(\tfrac{1}{p_{i}}-\tfrac{s_{i}}{n}\big)%
_{+}<\tfrac{1}{p}\leq \tfrac{1}{t}\leq 1.
\end{equation*}%
The whole estimate is divided to two steps.

$\bullet $ \textbf{Step 1. }We consider the case $s_{i}=\frac{n}{p_{i}},$%
\textbf{\ }$i\in K$\textbf{\ and\ }$s_{i}<\frac{n}{p_{i}},$ $i\in
I\backslash K,\mathbf{\ }K\subsetneq I$\textbf{.} We have%
\begin{equation}
\sum_{i\in I\backslash \left( K\cup \left\{ 1\right\} \right) }\big(\tfrac{1%
}{p_{i}}-\tfrac{s_{i}}{n}\big)+\tfrac{1}{p_{1}}<\tfrac{1}{p}\leq \tfrac{1}{t}%
.  \label{cond4}
\end{equation}%
We decompose $K\cup \left\{ 1\right\} $ into the disjoint union of $K_{1}$
and $K_{2}$, where%
\begin{equation}
K_{1}\subseteq I_{1}\text{ and }K_{2}\subseteq I_{2}  \label{cond5}
\end{equation}%
and $I_{1},I_{2}$ are defined in Section 3 (it seem obviously that $%
K_{1}\subsetneq I_{1}$ if $K_{2}=I_{2}$ and $K_{2}\subsetneq I_{2}$ if $%
K_{1}=I_{1}$. Also $1\notin K$ and if $1\in I_{1}$ then we have $1\in K_{1}$%
). Due to some technical reasons, we split this step into three separate
cases.

$\bullet $ \textbf{Case 1.} $I_{1}\backslash K_{1}\neq \emptyset $\textbf{\
and }$I_{2}\backslash K_{2}\neq \emptyset $\textbf{.} We continue with the
following subcases.

$\bullet $ \textbf{Subcase 1.1. } Assume that%
\begin{equation*}
\tfrac{1}{p_{1}}+\sum_{i\in I_{1}\backslash K_{1}}\big(\tfrac{1}{p_{i}}-%
\tfrac{s_{i}}{n}\big)<\tfrac{1}{p}-\sum_{i\in I_{2}\backslash K_{2}}\big(%
\tfrac{1}{p_{i}}-\tfrac{s_{i}}{n}\big)<\tfrac{1}{p_{1}}-\tfrac{s_{1}}{n}%
+\sum_{i\in I_{1}\backslash K_{1}}\tfrac{1}{p_{i}}.
\end{equation*}%
We put%
\begin{equation*}
\tfrac{1}{p}=\tfrac{1}{p_{1}}-\tfrac{s_{1}}{n}+\sum_{i\in I_{1}\backslash
K_{1}}\tfrac{1}{p_{i}}+\sum_{i\in I_{2}\backslash K_{2}}\tfrac{1}{t_{i}}%
,\quad \beta _{i}=s_{i}-\tfrac{n}{p_{i}}+\tfrac{n}{t_{i}},\quad i\in
I_{2}\backslash K_{2},
\end{equation*}%
where 
\begin{equation*}
\tfrac{s_{1}-\sum_{i\in I_{1}\backslash K_{1}}s_{i}}{n\left\vert
I_{2}\backslash K_{2}\right\vert }<\tfrac{1}{t_{i}}-\big(\tfrac{1}{p_{i}}-%
\tfrac{s_{i}}{n}\big)<0,\quad i\in I_{2}\backslash K_{2}.
\end{equation*}%
These are possible since both $I_{1}\backslash K_{1}$ and $I_{2}\backslash
K_{2}$ are not empty. Thanks to Lemma \ref{ortholemma2} it follows%
\begin{equation*}
\big\|\Pi _{2}(f_{1},f_{2},...,f_{m})\mid B_{p,\min \left( p,q\right)
}^{s_{1}}\big\|\leq c\Big\|\Big\{\sum^{j}\Big(\prod_{i=1}^{m}\Delta
_{k_{i}}f_{i}\Big)\Big\}\mid \ell _{\min \left( p,q\right) }^{s_{1}}\left(
L_{p}\right) \Big\|.
\end{equation*}%
First we see that, Lemma \ref{estimates}/(iii)-(iv) yields 
\begin{equation*}
\big\|\tilde{Q}_{j}f_{1}\big\|_{\tilde{p}_{1}}\leq c\text{ }u_{j}\big\|%
f_{1}\mid B_{p_{1},\infty }^{s_{1}}\big\|,
\end{equation*}%
with $\frac{1}{\tilde{p}_{1}}=\tfrac{1}{p_{1}}-\tfrac{s_{1}}{n}$ and%
\begin{equation}
u_{j}=\left\{ 
\begin{array}{ccc}
1, & \text{if} & 1\in I_{1}, \\ 
j+1, & \text{if} & 1\in I_{2}.%
\end{array}%
\right.  \label{cond6}
\end{equation}%
The H\"{o}lder inequality yields%
\begin{align}
& 2^{js_{1}}\Big\|\sum^{j}\Big(\prod_{i=1}^{m}\Delta _{k_{i}}f_{i}\Big)\Big\|%
_{p}  \label{cond7} \\
\leq & c\text{ }2^{js_{1}}u_{j}\big\|f_{1}\mid B_{p_{1},\infty }^{s_{1}}%
\big\|\prod_{i\in I_{1}\backslash K_{1}}\big\|\Delta _{j}f_{i}\big\|%
_{p_{i}}\prod_{i\in I_{2}\backslash K_{2}}\big\|Q_{j}f_{i}\big\|%
_{t_{i}}\prod_{i\in K}\big\|\tilde{Q}_{j}f_{i}\big\|_{\infty },  \notag
\end{align}%
for any $j\in \mathbb{N}_{0}$ with $c>0$ independent of\ $j$.\ Since $\beta
_{i}<0$ for any $i\in I_{2}\backslash K_{2}$, then Lemma \ref{estimates}%
/(ii)-(iii) gives%
\begin{equation*}
\big\|Q_{j-N}f_{i}\big\|_{t_{i}}\leq c\text{ }2^{-\beta _{i}j}\big\|%
f_{i}\mid B_{t_{i},\infty }^{\beta _{i}}\big\|,\quad i\in I_{2}\backslash
K_{2}
\end{equation*}%
and%
\begin{equation}
\big\|\tilde{Q}_{j}f_{i}\big\|_{\infty }\leq c\big\|f_{i}\mid
B_{p_{i},\infty }^{n/p_{i}}\big\|\times \left\{ 
\begin{array}{ccc}
1, & \text{if} & i\in K_{1}\backslash \left\{ 1\right\} , \\ 
j+1, & \text{if} & i\in K_{2}\backslash \left\{ 1\right\} .%
\end{array}%
\right.  \label{cond9}
\end{equation}%
Therefore the expression $\mathrm{\eqref{cond7}}$ is bounded by%
\begin{equation*}
c\text{ }2^{\gamma nj}u_{j}^{\prime }\big\|f_{1}\mid B_{p_{1},\infty
}^{s_{1}}\big\|\prod_{i\in I_{1}\backslash K_{1}}\big\|f_{i}\mid
B_{p_{i},\infty }^{s_{i}}\big\|\prod_{i\in I_{2}\backslash K_{2}}\big\|%
f_{i}\mid B_{t_{i},\infty }^{\beta _{i}}\big\|\prod_{i\in K}\big\|f_{i}\mid
B_{p_{i},\infty }^{n/p_{i}}\big\|,
\end{equation*}%
where for $j\in \mathbb{N}_{0}$,%
\begin{equation}
u_{j}^{\prime }=\left\{ 
\begin{array}{ccc}
\left( j+1\right) ^{\left\vert K_{2}\right\vert }, & \text{if} & 1\in I_{1}%
\text{ or }1\in K_{2}, \\ 
\left( j+1\right) ^{\left\vert K_{2}\right\vert +1}, & \text{if} & 1\notin
K_{2}\text{ and }1\in I_{2}%
\end{array}%
\right.  \label{cond10}
\end{equation}%
and 
\begin{equation*}
\gamma =\sum_{i\in I\backslash \left( K\cup \left\{ 1\right\} \right) }\big(%
\tfrac{1}{p_{i}}-\tfrac{s_{i}}{n}\big)+\tfrac{1}{p_{1}}-\tfrac{1}{p}.
\end{equation*}%
Using the embeddings%
\begin{equation}
F_{p,q}^{s_{1}}\hookrightarrow B_{p,\min \left( p,q\right) }^{s_{1}},\quad
F_{p_{1},q}^{s_{1}}\hookrightarrow B_{p_{1},\infty }^{s_{1}}  \label{cond11}
\end{equation}%
and 
\begin{equation*}
B_{p_{i},\infty }^{s_{i}}\hookrightarrow B_{t_{i},\infty }^{\beta
_{i}},\quad i\in I_{2}\backslash K_{2},
\end{equation*}%
we obtain that the $F_{p,q}^{s_{1}}$-norm of $\Pi _{2}\left(
f_{1},f_{2},...,f_{m}\right) $ can be estimated by%
\begin{equation*}
c\big\|f_{1}\mid F_{p_{1},q}^{s_{1}}\big\|\prod_{i\in I\backslash (K\cup
\{1\})}\big\|f_{i}\mid B_{p_{i},\infty }^{s_{i}}\big\|\prod_{i\in K}\big\|%
f_{i}\mid B_{p_{i},\infty }^{n/p_{i}}\big\|,
\end{equation*}%
in view of the fact that $\big\{2^{\gamma nj}u_{j}^{\prime }\big\}\in \ell
_{\min \left( p,q\right) }$.

$\bullet $ \textbf{Subcase 1.2.} We consider the case%
\begin{equation}
\tfrac{1}{p_{1}}-\tfrac{s_{1}}{n}+\sum_{i\in I\backslash \left( K\cup
\left\{ 1\right\} \right) }\tfrac{1}{p_{i}}-\sum_{i\in I_{2}\backslash K_{2}}%
\tfrac{s_{i}}{n}\leq \tfrac{1}{p}\leq \tfrac{1}{t}.  \label{cond12}
\end{equation}%
First, by H\"{o}lder's inequality and Lemma \ref{estimates}/(ii)-(iii), we
get%
\begin{align*}
& 2^{js_{1}}\Big\|\sum^{j}\Big(\prod_{i=1}^{m}\Delta _{k_{i}}f_{i}\Big)\Big\|%
_{t} \\
\leq & c\text{ }2^{js_{1}}u_{j}\big\|f_{1}\mid B_{p_{1},\infty }^{s_{1}}%
\big\|\prod_{i\in I_{1}\backslash K_{1}}\big\|\Delta _{j}f_{i}\big\|%
_{p_{i}}\prod_{i\in I_{2}\backslash K_{2}}\big\|Q_{j}f_{i}\big\|%
_{p_{i}}\prod_{i\in K}\big\|\tilde{Q}_{j}f_{i}\big\|_{p_{i}} \\
\leq & c\text{ }2^{j\varrho }u_{j}\prod_{i\in I\backslash K}\big\|f_{i}\mid
B_{p_{i},\infty }^{s_{i}}\big\|\prod_{i\in K}\big\|f_{i}\mid B_{p_{i},\infty
}^{n/p_{i}}\big\|,\quad j\in \mathbb{N}_{0},
\end{align*}%
where $c>0$ independent of\ $j$ and%
\begin{equation*}
\varrho =\left\{ 
\begin{array}{ccc}
-\sum\limits_{i\in I_{1}\backslash \left\{ 1\right\} }s_{i}, & \text{if} & 
1\in I_{1}, \\ 
s_{1}-\sum\limits_{i\in I_{1}\backslash \left\{ 1\right\} }s_{i}, & \text{if}
& 1\notin I_{1}.%
\end{array}%
\right.
\end{equation*}%
Since, in view of the fact that $\left\vert I_{1}\right\vert \geq 2$, 
\begin{equation*}
\big\{2^{j\varrho }u_{j}\big\}\in \ell _{\min \left( p,q\right) },
\end{equation*}%
and then by $\mathrm{\eqref{cond11}}$ we conclude the desired estimate. Now
let 
\begin{equation*}
\tfrac{1}{t_{1}}=\tfrac{1}{p_{1}}-\tfrac{s_{1}}{n}+\sum_{i\in
I_{1}\backslash K_{1}}\tfrac{1}{p_{i}}+\sum_{i\in I_{2}\backslash K_{2}}%
\tfrac{1}{t_{i}},
\end{equation*}%
where 
\begin{equation*}
\tfrac{1}{t_{i}}=\tfrac{1}{p_{i}}-\tfrac{s_{i}}{n},\quad i\in
I_{2}\backslash K_{2}.
\end{equation*}%
The H\"{o}lder inequality yields the estimate $\mathrm{\eqref{cond7}}$ (with 
$t_{1}$ in place of $p$). Lemma \ref{estimates}/(iii) gives 
\begin{equation}
\prod_{i\in I_{2}\backslash K_{2}}\big\|Q_{j}f_{i}\big\|_{t_{i}}\leq c\left(
j+1\right) ^{\left\vert I_{2}\backslash K_{2}\right\vert }\prod_{i\in
I_{2}\backslash K_{2}}\big\|f_{i}\mid B_{p_{i},\infty }^{s_{i}}\big\|.
\label{cond13}
\end{equation}%
The last estimate and $\mathrm{\eqref{cond9}}$ yield that 
\begin{equation*}
2^{js_{1}}\Big\|\sum^{j}\Big(\prod_{i=1}^{m}\Delta _{k_{i}}f_{i}\Big)\Big\|%
_{t_{1}}\leq c\text{ }2^{\gamma _{2}nj}\left( j+1\right) ^{\left\vert
I_{2}\backslash K_{2}\right\vert }u_{j}^{\prime }\prod_{i\in I\backslash K}%
\big\|f_{i}\mid B_{p_{i},\infty }^{s_{i}}\big\|\prod_{i\in K}\big\|f_{i}\mid
B_{p_{i},\infty }^{n/p_{i}}\big\|,
\end{equation*}%
where 
\begin{equation*}
\gamma _{2}=s_{1}-\sum_{i\in I_{1}\backslash K_{1}}s_{i}
\end{equation*}%
and $\big\{u_{j}^{\prime }\big\}$ is defined in $\mathrm{\eqref{cond10}}$.
We use the fact that 
\begin{equation*}
\big\{2^{\gamma _{2}nj}\left( j+1\right) ^{\left\vert I_{2}\backslash
K_{2}\right\vert }u_{j}^{\prime }\big\}\in \ell _{\min \left( p,q\right) },
\end{equation*}%
then the desired estimate can be obtained by the embeddings $\mathrm{%
\eqref{cond11}}$. Let $f_{i}\in B_{p_{i},\infty }^{s_{i}},i\in \left\{
2,...,m\right\} $. Hence 
\begin{equation*}
\Pi _{2}(\cdot ,f_{2},...,f_{m}):F_{p_{1},q}^{s_{1}}\longrightarrow
F_{t,q}^{s_{1}}\quad \text{and}\quad \Pi _{2}(\cdot
,f_{2},...,f_{m}):F_{p_{1},q}^{s_{1}}\longrightarrow F_{t_{1},q}^{s_{1}}.
\end{equation*}%
To cover all $p$ satisfying $\mathrm{\eqref{cond12}}$ we apply Theorem \ref%
{interpolation} and Remark \ref{interpolation1}. That completes the proof of
this subcase.

$\bullet $ \textbf{Case 2. }$I_{1}\backslash K_{1}\neq \emptyset $\textbf{\
\ and }$I_{2}\backslash K_{2}=\emptyset $\textbf{.} This case yields $K\cup
\left\{ 1\right\} =I_{2}\cup K_{1}$, where $K_{1}\subsetneq I_{1}$.
Furthermore in view of \eqref{cond4}, we have\textrm{\ } 
\begin{equation*}
\sum_{i\in I_{1}\backslash K_{1}}\big(\tfrac{1}{p_{i}}-\tfrac{s_{i}}{n}\big)+%
\tfrac{1}{p_{1}}<\tfrac{1}{p}\leq \tfrac{1}{t}.
\end{equation*}

$\bullet $ \textbf{Subcase 2.1. }Assume that 
\begin{equation*}
\sum_{i\in I_{1}\backslash K_{1}}\big(\tfrac{1}{p_{i}}-\tfrac{s_{i}}{n}\big)+%
\tfrac{1}{p_{1}}<\tfrac{1}{p}\leq \tfrac{1}{p_{1}}+\sum_{i\in
I_{1}\backslash K_{1}}\tfrac{1}{p_{i}}.
\end{equation*}%
We set 
\begin{equation*}
\tfrac{1}{p}=\tfrac{1}{p_{1}}+\sum_{i\in I_{1}\backslash K_{1}}\tfrac{1}{%
t_{i}},
\end{equation*}%
where 
\begin{equation*}
\tfrac{1}{p_{i}}-\tfrac{s_{i}}{n}<\tfrac{1}{t_{i}}\leq \tfrac{1}{p_{i}}%
,\quad i\in I_{1}\backslash K_{1}.
\end{equation*}%
The H\"{o}lder inequality, $\mathrm{\eqref{cond9}}$ and Lemma \ref{estimates}%
/(iii), yield%
\begin{align*}
& 2^{js_{1}}\Big\|\sum^{j}\Big(\prod_{i=1}^{m}\Delta _{k_{i}}f_{i}\Big)\Big\|%
_{p} \\
\leq & c\text{ }2^{js_{1}}u_{j}\big\|f_{1}\mid B_{p_{1},\infty }^{s_{1}}%
\big\|\prod_{i\in I_{1}\backslash K_{1}}\big\|\Delta _{j}f_{i}\big\|%
_{t_{i}}\prod_{i\in K}\big\|\tilde{Q}_{j}f_{i}\big\|_{\infty } \\
\leq & c\text{ }2^{js_{1}}u_{j}^{\prime \prime }\big\|f_{1}\mid
B_{p_{1},\infty }^{s_{1}}\big\|\prod_{i\in I_{1}\backslash K_{1}}\big\|%
\Delta _{j}f_{i}\big\|_{t_{i}}\prod_{i\in K}\big\|f_{i}\mid B_{p_{i},\infty
}^{n/p_{i}}\big\| \\
\leq & c\text{ }2^{\gamma _{3}nj}u_{j}^{\prime \prime }\prod_{i\in \left(
I_{1}\backslash K_{1}\right) \cup \left\{ 1\right\} }\big\|f_{i}\mid
B_{p_{i},\infty }^{s_{i}}\big\|\prod_{i\in K}\big\|f_{i}\mid B_{p_{i},\infty
}^{n/p_{i}}\big\|,
\end{align*}%
where $c>0$ independent of\ $j$, $\left\{ u_{j}\right\} $ is defined in $%
\mathrm{\eqref{cond6}}$ and for $j\in \mathbb{N}_{0}$ 
\begin{equation*}
u_{j}^{\prime \prime }=\left\{ 
\begin{array}{ccc}
\left( j+1\right) ^{\left\vert I_{2}\right\vert -1} & \text{if} & 1\in I_{2}
\\ 
2^{-s_{1}j}\left( j+1\right) ^{\left\vert I_{2}\right\vert } & \text{if} & 
1\in I_{1}%
\end{array}%
\right.
\end{equation*}%
and 
\begin{equation*}
\gamma _{3}=\tfrac{1}{p_{1}}-\tfrac{1}{p}+\sum_{i\in I_{1}\backslash K_{1}}%
\big(\tfrac{1}{p_{i}}-\tfrac{s_{i}}{n}\big).
\end{equation*}%
We take $\ell _{\min (p,q)}$-quasi-norm and we conclude the desired estimate
using $\mathrm{\eqref{cond11}}$\textrm{, }in view of the fact that $\gamma
_{3}<0$.

$\bullet $ \textbf{Subcase 2.2. }Assume that%
\begin{equation*}
\tfrac{1}{p_{1}}+\sum_{i\in I_{1}\backslash K_{1}}\tfrac{1}{p_{i}}<\tfrac{1}{%
p}\leq \tfrac{1}{t}.
\end{equation*}%
This case can be covered by the interpolation\ method given in Subcase\ 1.2.

$\bullet $ \textbf{Case 3.} $I_{1}\backslash K_{1}=\emptyset $\textbf{\ and }%
$I_{2}\backslash K_{2}\neq \emptyset $\textbf{.} This case yields $K\cup
\left\{ 1\right\} =I_{1}\cup K_{2}$, where $K_{2}\subsetneq I_{2}$. We set $%
I_{2}=K_{2}\cup I_{3}$. Therefore, by \eqref{cond4}, we have%
\begin{equation*}
\tfrac{1}{p_{1}}+\sum_{i\in I_{3}\backslash \left\{ 1\right\} }\big(\tfrac{1%
}{p_{i}}-\tfrac{s_{i}}{n}\big)<\tfrac{1}{p}\leq \tfrac{1}{t}.
\end{equation*}

$\bullet $ \textbf{Subcase\ 3.1. }We consider the case%
\begin{equation*}
\tfrac{1}{p_{1}}<\tfrac{1}{p}-\sum_{i\in I_{3}\backslash \left\{ 1\right\} }%
\big(\tfrac{1}{p_{i}}-\tfrac{s_{i}}{n}\big)<\tfrac{1}{p_{1}}-\tfrac{s_{1}}{n}%
+\sum_{i\in I_{1}\backslash \left\{ 1\right\} }\tfrac{1}{p_{i}}.
\end{equation*}%
Let $1<v<p$ be such that 
\begin{equation*}
\tfrac{1}{v}=\tfrac{1}{p_{1}}-\tfrac{s_{1}}{n}+\sum_{i\in I_{3}\backslash
\left\{ 1\right\} }\tfrac{1}{d_{i}}+\sum_{i\in I_{1}\backslash \left\{
1\right\} }\tfrac{1}{t_{i}},
\end{equation*}%
where $t_{i}>p_{i},i\in I_{1}\backslash \left\{ 1\right\} \ $and 
\begin{equation*}
\tfrac{1}{d_{i}}=\tfrac{1}{p_{i}}-\tfrac{s_{i}}{n},\quad i\in
I_{3}\backslash \left\{ 1\right\} .
\end{equation*}%
We set $\sigma =s_{1}-\tfrac{n}{p}+\tfrac{n}{v}$. The H\"{o}lder inequality
and Lemma \ref{estimates}/(iii), yield%
\begin{align*}
& 2^{j\sigma }\Big\|\sum^{j}\Big(\prod_{i=1}^{m}\Delta _{k_{i}}f_{i}\Big)%
\Big\|_{v} \\
\leq & c\text{ }u_{j}2^{j\sigma }\big\|f_{1}\mid B_{p_{1},\infty }^{s_{1}}%
\big\|\prod_{i\in I_{1}\backslash \left\{ 1\right\} }\big\|\Delta _{j}f_{i}%
\big\|_{t_{i}}\prod_{i\in K_{2}\backslash \left\{ 1\right\} }\big\|Q_{j}f_{i}%
\big\|_{\infty }\prod_{i\in I_{3}\backslash \left\{ 1\right\} }\big\|%
Q_{k}f_{i}\big\|_{d_{i}} \\
\leq & c\text{ }\eta _{j}\prod_{i=1}^{m}\big\|f_{i}\mid B_{p_{i},\infty
}^{s_{i}}\big\|,
\end{align*}%
where $c>0$ independent of\ $j$, $\left\{ u_{j}\right\} $ is defined in $%
\mathrm{\eqref{cond6}}$ and$\ $for $j\in \mathbb{N}_{0}$%
\begin{equation*}
\eta _{j}=2^{j(\tfrac{n}{p_{1}}+\sum_{i\in I_{3}\backslash \left\{ 1\right\}
}(\tfrac{n}{p_{i}}-s_{i})-\tfrac{n}{p})}\times \left\{ 
\begin{array}{ccc}
\left( j+1\right) ^{\left\vert I_{3}\right\vert }, & \text{if} & 1\in I_{1}%
\text{ or }1\in I_{3}, \\ 
\left( j+1\right) ^{\left\vert I_{3}\right\vert +1}, & \text{if} & 1\in K_{2}%
\end{array}%
\right. .
\end{equation*}%
Since $\left\{ \eta _{j}\right\} \in \ell _{p}$ and $\sigma >0$ we conclude
the desired estimate using $B_{v,p}^{\sigma }\hookrightarrow F_{p,q}^{s_{1}}$
and Lemma \ref{ortholemma2}.

$\bullet $ \textbf{Subcase 3.2. }We consider the case%
\begin{equation}
\sum_{i\in I_{1}\backslash \left\{ 1\right\} }\tfrac{1}{p_{i}}+\sum_{i\in
I_{3}\backslash \left\{ 1\right\} }\big(\tfrac{1}{p_{i}}-\tfrac{s_{i}}{n}%
\big)+\tfrac{1}{p_{1}}-\tfrac{s_{1}}{n}\leq \tfrac{1}{p}\leq \tfrac{1}{t}.
\label{new-cond}
\end{equation}%
Let $\tfrac{1}{t_{1}}$ be the real number given in the\ left-hand side of $%
\mathrm{\eqref{new-cond}}$. First Lemma \ref{estimates}/(iii), gives the
estimate $\mathrm{\eqref{cond13}}$ for $\infty $, $K_{2}\backslash \left\{
1\right\} $ and for $\big(\tfrac{1}{p_{i}}-\tfrac{s_{i}}{n}\big)^{-1}$, $%
I_{3}\backslash \left\{ 1\right\} $ in place of $t_{i}$, $I_{2}\backslash
K_{2}\ $respectively. The H\"{o}lder inequality, yields%
\begin{equation*}
2^{js_{1}}\Big\|\sum^{j}\Big(\prod_{i=1}^{m}\Delta _{k_{i}}f_{i}\Big)\Big\|%
_{t_{1}}\leq c\text{ }\omega _{j}\prod_{i=1}^{m}\big\|f_{i}\mid
B_{p_{i},\infty }^{s_{i}}\big\|,\quad j\in \mathbb{N}_{0},
\end{equation*}%
where $c>0$ independent of\ $j$, with 
\begin{equation*}
\omega _{j}=\left( j+1\right) ^{\left\vert I_{3}\backslash \left\{ 1\right\}
\right\vert +\left\vert K_{2}\backslash \left\{ 1\right\} \right\vert
}2^{j(s_{1}-\sum_{i\in I_{1}\backslash \left\{ 1\right\} }s_{i})}u_{j}
\end{equation*}%
for any $j\in \mathbb{N}_{0}$. Since $\left\{ \omega _{j}\right\} \in \ell
_{\min \left( p,q\right) }$, we conclude the desired estimate using $\mathrm{%
\eqref{cond11}}$. Let $f_{i}\in B_{p_{i},\infty }^{s_{i}},i\in \left\{
2,...,m\right\} $. By a simple modifications of arguments used in$\ $Subcase
1.2, we get 
\begin{equation*}
\Pi _{2}(\cdot ,f_{2},...,f_{m}):F_{p_{1},q}^{s_{1}}\longrightarrow
F_{t_{1},q}^{s_{1}}\quad \text{and}\quad \Pi _{2}(\cdot
,f_{2},...,f_{m}):F_{p_{1},q}^{s_{1}}\longrightarrow F_{t,q}^{s_{1}}.
\end{equation*}%
Hence by the same interpolation method given in Subcase 1.2, we get the
desired result.

$\bullet $ \textbf{Step 2. }We consider the case $s_{i}>\tfrac{n}{p_{i}}%
,i\in I_{4}\varsubsetneq I$\textbf{.} We have%
\begin{equation*}
I\backslash I_{4}=\{i:s_{i}\leq \tfrac{n}{p_{i}}\}=\tilde{I}.
\end{equation*}%
The desired estimation\ can be obtained by a simple modifications of
arguments used in Step 1, where we replace $I$ by $\tilde{I}$ and the fact
that%
\begin{equation*}
\big\|\tilde{Q}_{j}f_{i}\big\|_{\infty }\leq c\text{ }\big\|f_{i}\mid
B_{p_{i},\infty }^{s_{i}}\big\|
\end{equation*}%
for any $j\in \mathbb{N}_{0}$ and any $0<p_{i}\leq \infty $, $i\in I_{4}$,
where the positive constant $c$ is independent of $j$.

This finishes the proof of Theorem \ref{multi-positive-smothness}.
\end{proof}

\begin{remark}
The result of this subsection with $m=2$\ is given in \cite{DM2}.
\end{remark}

\subsection{Products in spaces with negative smoothness}

In this subsection we deal with the case 
\begin{equation}
s_{1}\leq 0<s_{2}\leq s_{3}\leq ...\leq s_{m}.  \label{Cond5.1}
\end{equation}%
The main result is the following statement.

\begin{theorem}
\label{multi-negative-smothness}\textit{Let }$1\leq p_{1}<\infty $, $%
0<p_{i}\leq \infty $ and $0<q\leq \infty ,i=2,...,m$\textit{.\ Suppose
further\ }$\mathrm{\eqref{cond2.1}}$, $\mathrm{\eqref{cond2.2}}$\textrm{, }$%
\mathrm{\eqref{Cond5.1}}$ \textit{and}%
\begin{equation*}
s_{1}+s_{2}>0
\end{equation*}%
\textit{Then}%
\begin{equation*}
F_{p_{1},q}^{s_{1}}\cdot B_{p_{2},\infty }^{s_{2}}\cdot ...\cdot
B_{p_{m},\infty }^{s_{m}}\hookrightarrow F_{p,q}^{s_{1}}
\end{equation*}%
\textit{holds.}
\end{theorem}

\begin{corollary}
\label{Cor1}\textit{Under the hypotheses of Theorem }\ref%
{multi-negative-smothness}\textit{, then it holds }%
\begin{equation*}
F_{p_{1},q}^{s_{1}}\cdot F_{p_{2},\infty }^{s_{2}}\cdot ...\cdot
F_{p_{m},\infty }^{s_{m}}\hookrightarrow F_{p,q}^{s_{1}}.
\end{equation*}
\end{corollary}

The proof of Corollary $\mathrm{\ref{Cor1}}$ is immediate because $%
F_{p_{i},\infty }^{s_{i}}\hookrightarrow B_{p_{i},\infty }^{s_{i}}$, $%
i=2,...,m$.

\begin{remark}
We note that Corollary \ref{Cor1} is given in \cite[Theorem \ 2.4.5]{S} and 
\cite[ Theorem \ 4.5.2/1]{RS}.
\end{remark}

\begin{remark}
As in the preceding subsection, the case $m=2$ (with $0<p,p_{1}<\infty $) in
this section was treated in\cite{DM2}.
\end{remark}

\begin{proof}[\textbf{Proof of Theorem \protect\ref{multi-negative-smothness}%
.}]
In view of the estimation of $\Pi _{1,k}(f_{1},f_{2},...,f_{m})$ in the
proof of Theorem \ref{multi-positive-smothness}, it remains to consider $\Pi
_{2}\left( f_{1},f_{2},...,f_{m}\right) $. Indeed, we see that we need only
to study the case $k\neq 1$. Let $\tfrac{1}{p}=\tfrac{1}{p_{1}}%
+\sum_{i=2}^{m}\tfrac{1}{t_{i}}$ where 
\begin{equation*}
\big(\tfrac{1}{p_{i}}-\tfrac{s_{i}}{n}\big)_{+}<\tfrac{1}{t_{i}}\leq \tfrac{1%
}{p_{i}},\quad i\in I\backslash \{1\}.
\end{equation*}%
Our estimate follows by using the estimate $\mathrm{\eqref{estimateQ}}$, 
\begin{equation*}
\sup_{j\geq N}\big\|\Delta _{j}f_{k}\big\|_{t_{k}}\leq c\big\|f_{k}\mid
B_{p_{k},\infty }^{s_{k}}\big\|.
\end{equation*}%
and the fact that 
\begin{equation*}
\Big\|\Big\{2^{js_{1}}Q_{_{j-N}}f_{1}\Big\}_{j\geq N}\mid L_{p}(\ell _{q})%
\Big\|\lesssim \big\|f_{1}\mid F_{p_{1},q}^{s_{1}}\big\|.
\end{equation*}

\textbf{Estimation\ of }$\Pi _{2}(f_{1},f_{2},...,f_{m})$. Let $K$ be us in $%
\mathrm{\eqref{cond5}}$. For simplicity we put $\sum_{i\in K}\tfrac{1}{r}=0\ 
$if$\ K=\emptyset $. From $\mathrm{\eqref{cond2.1}}$ and $\mathrm{%
\eqref{cond2.2}}$, we get 
\begin{equation}
\tfrac{s_{1}}{n}+\sum_{i=1}^{m}\big(\tfrac{1}{p_{i}}-\tfrac{s_{i}}{n}\big)%
_{+}<\tfrac{1}{p}\leq \sum_{i=1}^{m}\tfrac{1}{p_{i}}\leq 1.
\label{new-cond1}
\end{equation}%
We split estimation into several distinct parts;

$\bullet $\textbf{\ Step 1. }We consider the case $s_{i}=\tfrac{n}{p_{i}}%
,i\in K$\ and\ $s_{i}<\tfrac{n}{p_{i}},i\in I\backslash K,K\subsetneq I$%
\textbf{. }We have 
\begin{equation*}
\tfrac{s_{1}}{n}+\big(\tfrac{1}{p_{1}}-\tfrac{s_{1}}{n}\big)_{+}=\tfrac{1}{%
p_{1}}.
\end{equation*}%
Again, we decompose $K\cup \left\{ 1\right\} $ into the disjoint union of $%
K_{1}$ and $K_{2}$, where $K_{1}\subseteq I_{1}$ and $K_{2}\subseteq I_{2}$,
see $\mathrm{\eqref{cond5}}$. We are forced to consider the following cases
separately:

$\bullet $ \textbf{Case 1.\ }$I_{1}\backslash K_{1}\neq \emptyset $\textbf{\
and }$I_{2}\backslash K_{2}\neq \emptyset $\textbf{. }We will divide this
case into the following two subcases.

$\bullet $ \textbf{Subcase 1.1. }Here we deal with%
\begin{equation}
\tfrac{1}{p_{1}}+\sum_{i\in I_{1}\backslash K_{1}}\tfrac{1}{p_{i}}%
+\sum_{i\in I_{2}\backslash K_{2}}\big(\tfrac{1}{p_{i}}-\tfrac{s_{i}}{n}\big)%
\leq \tfrac{1}{p}\leq \tfrac{1}{t},  \label{Cond5.7.1}
\end{equation}%
where $t=\big(\sum_{i=1}^{m}\tfrac{1}{p_{i}}\big)^{-1}$. We choose $\theta $
such that $-s_{1}<\theta <s_{2}$. This guarantees%
\begin{equation}
B_{p,\infty }^{s_{1}+\theta }\hookrightarrow F_{p,q}^{s_{1}}.
\label{Cond5.7}
\end{equation}%
Thanks to Lemma \ref{ortholemma2} it follows%
\begin{equation*}
\big\|\Pi _{2}(f_{1},f_{2},...,f_{m})\mid B_{p,\infty }^{s_{1}+\theta }\big\|%
\leq c\Big\|\Big\{\sum^{j}\Big(\prod_{i=1}^{m}\Delta _{k_{i}}f_{i}\Big)\Big\}%
\mid \ell _{\infty }^{s_{1}+\theta }\left( L_{p}\right) \Big\|.
\end{equation*}%
First let%
\begin{equation}
\tfrac{1}{p_{0}}=\tfrac{1}{p_{1}}+\sum_{i\in I_{1}\backslash K_{1}}\tfrac{1}{%
p_{i}}+\sum_{i\in I_{2}\backslash K_{2}}\tfrac{1}{t_{i}},  \label{Cond5.8}
\end{equation}%
where 
\begin{equation*}
t_{i}=\big(\tfrac{1}{p_{i}}-\tfrac{s_{i}}{n}\big)^{-1},\quad i\in
I_{2}\backslash K_{2}.
\end{equation*}%
The same arguments used in Subcase 1.2 in the proof of Theorem \ref%
{multi-positive-smothness}, yield%
\begin{equation*}
2^{j\left( s_{1}+\theta \right) }\Big\|\sum^{j}\Big(\prod_{i=1}^{m}\Delta
_{k_{i}}f_{i}\Big)\Big\|_{p_{0}}\leq c\text{ }\varepsilon _{j}\prod_{i\in
I\backslash K}\big\|f_{i}\mid B_{p_{i},\infty }^{s_{i}}\big\|\prod_{i\in K}%
\big\|f_{i}\mid B_{p_{i},\infty }^{n/p_{i}}\big\|,\quad j\in \mathbb{N}_{0},
\end{equation*}%
where the positive constant $c$ is independent of $j$ and 
\begin{equation*}
\varepsilon _{j}=2^{j(\theta -\sum_{i\in I_{1}\backslash K_{1}}s_{i})}\left(
j+1\right) ^{\left\vert I_{2}\backslash \left\{ 1\right\} \right\vert
},\quad j\in \mathbb{N}_{0}.
\end{equation*}%
The choice of $\theta $ ensure that $\sum_{i\in I_{1}\backslash
K_{1}}s_{i}-\theta >0$, this yields $\left\{ \varepsilon _{j}\right\} \in
\ell _{\infty }$. The above estimate, combined with the embedding%
\begin{equation}
F_{p_{1},q}^{s_{1}}\hookrightarrow B_{p_{1},\infty }^{s_{1}},
\label{Cond5.9}
\end{equation}%
give 
\begin{equation*}
\Pi _{2}(\cdot ,f_{2},...,f_{m}):F_{p_{1},q}^{s_{1}}\longrightarrow
F_{p_{0},q}^{s_{1}}.
\end{equation*}%
By the same technical above we can prove that 
\begin{equation*}
\Pi _{2}(\cdot ,f_{2},...,f_{m}):F_{p_{1},q}^{s_{1}}\longrightarrow
F_{t,q}^{s_{1}},
\end{equation*}%
where $f_{i}\in B_{p_{i},\infty }^{s_{i}},i\in \left\{ 2,...,m\right\} $. To
cover all $p$ satisfying $\mathrm{\eqref{Cond5.7.1}}$ we apply Theorem \ref%
{interpolation} and Remark \ref{interpolation1}.

$\bullet $ \textbf{Subcase 1.2. }We consider the case when $p$ in the
following range:\textbf{\ }%
\begin{equation*}
\sum_{i\in I_{2}\backslash K_{2}}\tfrac{s_{i}}{n}<\tfrac{1}{p_{1}}%
+\sum_{i\in I\backslash \left( K\cup \left\{ 1\right\} \right) }\tfrac{1}{%
p_{i}}-\tfrac{1}{p}<\sum_{i\in I\backslash \left( K\cup \left\{ 1\right\}
\right) }\tfrac{s_{i}}{n}.
\end{equation*}%
We recall that $p_{0}$ is given in $\mathrm{\eqref{Cond5.8}}$. Let $\alpha
>0 $ be such that 
\begin{equation*}
\tfrac{n}{p_{0}}-\tfrac{n}{p}<\sum_{i\in I_{1}\backslash K_{1}}s_{i}-\alpha
\quad \text{and}\quad \alpha <\sum_{i\in I_{1}\backslash K_{1}}s_{i}+s_{1}.
\end{equation*}%
We see that 
\begin{equation*}
B_{p_{0},\infty }^{\sum_{i\in I_{1}\backslash K_{1}}s_{i}+s_{1}-\alpha
}\hookrightarrow B_{p_{0},\min \left( p,q\right) }^{s_{1}+\tfrac{n}{p_{0}}-%
\tfrac{n}{p}}\hookrightarrow B_{p,\min \left( p,q\right)
}^{s_{1}}\hookrightarrow F_{p,q}^{s_{1}}.
\end{equation*}%
Using the same technical used above, it is not hard to obtain the desired
estimate.

$\bullet $ \textbf{Case 2. }$I_{1}\backslash K_{1}\neq \emptyset $\ and $%
I_{2}\backslash K_{2}=\emptyset $. Then $K\cup \left\{ 1\right\} =I_{2}\cup
\left\{ 1\right\} \cup I_{3}$ where $I_{3}\subsetneq I_{1}$. This case can
be covered by the same arguments as in Case 1. So we omit the details.

$\bullet $ \textbf{Case 3. }$I_{1}\backslash K_{1}=\emptyset $\ and $%
I_{2}\backslash K_{2}\neq \emptyset $. To prove we additionally do it into
the two Substeps 3.1 and 3.2.

$\bullet $ \textbf{Subcase 3.1. }We begin by the case when $\ p$ in the
following range:\textbf{\ }%
\begin{equation*}
\tfrac{1}{p_{1}}<\tfrac{1}{p}-\sum_{i\in I_{2}\backslash K_{2}}\big(\tfrac{1%
}{p_{i}}-\tfrac{s_{i}}{n}\big)<\tfrac{1}{p_{1}}+\tfrac{s_{1}}{n}+\sum_{i\in
I_{1}\backslash \left\{ 1\right\} }\tfrac{1}{p_{i}}.
\end{equation*}

$\bullet $ $s_{1}>\tfrac{n}{p}-n$. Let $v>0$ be such that%
\begin{equation*}
\tfrac{1}{p}-\tfrac{s_{1}}{n}<\tfrac{1}{v}<\tfrac{1}{p_{1}}+\sum_{i\in
I_{1}\backslash \left\{ 1\right\} }\tfrac{1}{p_{i}}+\sum_{i\in
I_{2}\backslash K_{2}}\big(\tfrac{1}{p_{i}}-\tfrac{s_{i}}{n}\big).
\end{equation*}%
We set 
\begin{equation*}
\tfrac{1}{v}=\tfrac{1}{p_{1}}+\sum_{i\in I_{2}\backslash K_{2}}\tfrac{1}{%
d_{i}}+\sum_{i\in I_{1}\backslash \left\{ 1\right\} }\tfrac{1}{t_{i}},
\end{equation*}%
where $t_{i}>p_{i}$,\ $i\in I_{1}\backslash \left\{ 1\right\} $ and $\tfrac{1%
}{d_{i}}=\tfrac{1}{p_{i}}-\tfrac{s_{i}}{n}$, $i\in I_{2}\backslash K_{2}$.
With minor technical changes given in Subcase 3.1 in the proof of Theorem %
\ref{multi-positive-smothness}, we can get the desired result.

$\bullet $ $s_{1}\leq \tfrac{n}{p}-n$. Let $v>0$ be such that 
\begin{equation}
\tfrac{1}{v}=\tfrac{1}{p_{1}}+\sum_{i\in I_{1}\backslash \left\{ 1\right\} }%
\tfrac{1}{p_{i}}+\sum_{i\in I_{2}\backslash K_{2}}\tfrac{1}{t_{i}},\quad
\sigma =s_{1}-\tfrac{n}{p}+\tfrac{n}{v},  \label{sigma}
\end{equation}%
where 
\begin{equation*}
\big(\tfrac{1}{p_{i}}-\tfrac{s_{i}}{n}\big)^{-1}>t_{i}>p_{i},\quad i\in
I_{2}\backslash K_{2}.
\end{equation*}%
We have%
\begin{equation}
B_{v,p}^{\tfrac{n}{v}-n+\theta }\hookrightarrow B_{v,p}^{\tfrac{n}{v}%
-n}\hookrightarrow B_{v,p}^{\sigma }\hookrightarrow F_{p,q}^{s_{1}},
\label{Cond5.20}
\end{equation}%
where $\theta >n-\tfrac{n}{v}$. Therefore\ for any $j\in \mathbb{N}_{0}$,%
\begin{align*}
& 2^{j(\tfrac{n}{v}-n+\theta )}\Big\|\sum^{j}\Big(\prod_{i=1}^{m}\Delta
_{k_{i}}f_{i}\Big)\Big\|_{v} \\
\leq & c\text{ }2^{j\nu }u_{j}^{\prime }\big\|f_{1}\mid B_{p_{1},\infty
}^{s_{1}}\big\|\prod_{i\in I_{2}\backslash K_{2}}\big\|f_{i}\mid
B_{p_{i},\infty }^{s_{i}}\big\|\prod_{i\in I_{1}\cup K_{2}}\big\|f_{i}\mid
B_{p_{i},\infty }^{n/p_{i}}\big\|,
\end{align*}%
where the positive constant $c$ is independent of $j$, $\big\{u_{j}^{\prime }%
\big\}$ is defined by $\mathrm{\eqref{cond10}}$ and%
\begin{equation*}
\nu =\tfrac{n}{v}-n-s_{1}-\sum_{i\in I_{1}\backslash \left\{ 1\right\} }%
\tfrac{n}{p_{i}}+\theta .
\end{equation*}%
The choice of $\theta $ such that $\nu <0$, yields $\big\{2^{j\nu }\cdot
u_{j}^{\prime }\big\}\in \ell _{p}$. We conclude the desired estimate using $%
\mathrm{\eqref{Cond5.9}}$ and $\mathrm{\eqref{Cond5.20}}$.

$\bullet $ \textbf{Subcase 3.2. }We consider the case%
\begin{equation*}
\sum_{i\in I_{1}\backslash \left\{ 1\right\} }\tfrac{1}{p_{i}}+\sum_{i\in
I_{2}\backslash K_{2}}\big(\tfrac{1}{p_{i}}-\tfrac{s_{i}}{n}\big)+\tfrac{1}{%
p_{1}}+\tfrac{s_{1}}{n}\leq \tfrac{1}{p}\leq \tfrac{1}{t}.
\end{equation*}%
First let%
\begin{equation*}
\tfrac{1}{v}=\tfrac{1}{p_{1}}+\sum_{i\in I_{1}\backslash \left\{ 1\right\} }%
\big(\tfrac{1}{p_{i}}+\tfrac{s_{1}}{\left\vert I_{1}\backslash \left\{
1\right\} \right\vert n}\big)+\sum_{i\in I_{2}\backslash K_{2}}\big(\tfrac{1%
}{p_{i}}-\tfrac{s_{i}}{n}\big).
\end{equation*}%
Here we observe that $\tfrac{1}{p_{i}}+\tfrac{s_{1}}{\left\vert
I_{1}\backslash \left\{ 1\right\} \right\vert n}>0$ for any $i\in
I_{1}\backslash \left\{ 1\right\} $, since $s_{1}+s_{i}>0$ for any $i\in
I\backslash \left\{ 1\right\} $. First Lemma \ref{estimates}/(iii), gives 
\begin{equation*}
\prod_{i\in J_{2}}\big\|Q_{j}f_{i}\big\|_{\infty }\leq c\text{ }%
u_{j}^{\prime }\prod_{i\in K_{2}}\big\|f_{i}\mid B_{p_{i},\infty }^{s_{i}}%
\big\|,\quad j\in \mathbb{N}_{0}.
\end{equation*}%
The H\"{o}lder inequality and Lemma \ref{estimates}/(iii), yield%
\begin{equation*}
2^{j\left( s_{1}+\theta \right) }\Big\|\sum^{j}\Big(\prod_{i=1}^{m}\Delta
_{k_{i}}f_{i}\Big)\Big\|_{v}\leq c\text{ }u_{j}^{\prime }2^{j(\theta
-\sum_{i\in I_{1}\backslash \left\{ 1\right\} }n/p_{i})}\prod_{i=1}^{m}\big\|%
f_{i}\mid B_{p_{i},\infty }^{s_{i}}\big\|,\quad j\in \mathbb{N}_{0},
\end{equation*}%
where the positive constant $c$ is independent of $j\in \mathbb{N}_{0}$. By
taking 
\begin{equation}
-s_{1}<\theta <\sum_{i\in I_{1}\backslash \left\{ 1\right\} }\tfrac{n}{p_{i}}%
,  \label{theta}
\end{equation}%
which is possible because of $\left\vert I_{1}\backslash \left\{ 1\right\}
\right\vert \geq 1$ and we get $\big\{2^{j(\theta -\sum_{i\in
I_{1}\backslash \left\{ 1\right\} }n/p_{i})}\big\}\in \ell _{\min \left(
p,q\right) }$. Let $f_{i}\in B_{p_{i},\infty }^{s_{i}},i\in \left\{
2,...,m\right\} $. We conclude the desired estimate using $\mathrm{%
\eqref{Cond5.7}}$ and $\mathrm{\eqref{Cond5.9}}$. Therefore,%
\begin{equation*}
\Pi _{2}(\cdot ,f_{2},...,f_{m}):F_{p_{1},q}^{s_{1}}\longrightarrow
F_{v,q}^{s_{1}}.
\end{equation*}%
Now if $p=t$, by a simple modifications of arguments used in$\ $Subcase 1.2,
we get 
\begin{equation*}
\Pi _{2}(\cdot ,f_{2},...,f_{m}):F_{p_{1},q}^{s_{1}}\longrightarrow
F_{t,q}^{s_{1}}.
\end{equation*}%
Hence by the same interpolation method given above, we get the desired
result.

$\bullet $ \textbf{Step 2. }The following substeps are involved in finding
the desired estimate.

$\bullet $ \textbf{Substep 2.1.} $s_{i}=\tfrac{n}{p_{i}},i\in I\backslash
\left\{ 1\right\} $. In this case we have $p<p_{1}$. We distinguish between
the following two cases:

$\bullet $ \textbf{Case 1.}%
\begin{equation*}
\tfrac{1}{p_{1}}+\sum_{i\in I_{1}\backslash \left\{ 1\right\} }\tfrac{1}{%
p_{i}}\leq \tfrac{1}{p}\leq \tfrac{1}{t}.
\end{equation*}%
We set 
\begin{equation*}
\tfrac{1}{p}=\tfrac{1}{p_{1}}+\sum_{i\in I_{1}\backslash \left\{ 1\right\} }%
\tfrac{1}{p_{i}}+\sum_{i\in I_{2}\backslash \left\{ 1\right\} }\tfrac{1}{%
t_{i}},\quad t_{i}\geq p_{i},i\in I_{2}\backslash \left\{ 1\right\} .
\end{equation*}%
It is not hard to obtain%
\begin{equation*}
2^{j\left( s_{1}+\theta \right) }\Big\|\sum^{j}\Big(\prod_{i=1}^{m}\Delta
_{k_{i}}f_{i}\Big)\Big\|_{p}\leq c\text{ }2^{j(\theta -\sum_{i\in
I_{1}\backslash \left\{ 1\right\} }\tfrac{n}{p_{i}})}\big\|f_{1}\mid
B_{p_{1},\infty }^{s_{1}}\big\|\prod_{i=2}^{m}\big\|f_{i}\mid
B_{p_{i},\infty }^{n/p_{i}}\big\|
\end{equation*}%
for any $j\in \mathbb{N}_{0}$. We choose $\theta $ as in $\mathrm{%
\eqref{theta}}$ and using $\mathrm{\eqref{Cond5.7}}$ and $\mathrm{%
\eqref{Cond5.9}}$ we get the desired estimate.

$\bullet $ \textbf{Case 2.}%
\begin{equation*}
\tfrac{1}{p_{1}}<\tfrac{1}{p}<\tfrac{1}{p_{1}}+\sum_{i\in I_{1}\backslash
\left\{ 1\right\} }\tfrac{1}{p_{i}}.
\end{equation*}%
First we see that $s_{1}>\tfrac{n}{p_{1}}-n$, see $\mathrm{\eqref{new-cond1}}
$. Let $v>0$ be such that 
\begin{equation*}
\max \Big(\tfrac{1}{p},\tfrac{1}{p_{1}}-\tfrac{s_{1}}{n}\Big)<\tfrac{1}{v}<%
\tfrac{1}{p_{1}}+\sum_{i\in I_{1}\backslash \left\{ 1\right\} }\tfrac{1}{%
p_{i}}.
\end{equation*}%
We set 
\begin{equation*}
\tfrac{1}{v}=\tfrac{1}{p_{1}}+\sum_{i\in I_{1}\backslash \left\{ 1\right\} }%
\tfrac{1}{t_{i}},\quad t_{i}>p_{i},i\in I_{1}\backslash \left\{ 1\right\}
\end{equation*}%
and $\theta >0$ such that\ 
\begin{equation*}
-s_{1}-\tfrac{n}{v}<\theta -\tfrac{n}{p}<-\tfrac{n}{p_{1}}.
\end{equation*}%
Let $\sigma $ be as in $\mathrm{\eqref{sigma}}$. We can prove the following
estimates 
\begin{align*}
2^{j\left( \sigma +\theta \right) }\Big\|\sum^{j}\Big(\prod_{i=1}^{m}\Delta
_{k_{i}}f_{i}\Big)\Big\|_{v}\leq & c\text{ }2^{j\left( \sigma -s_{1}+\theta
\right) }\big\|f_{1}\mid B_{p_{1},\infty }^{s_{1}}\big\|\prod_{i\in
I_{1}\backslash \left\{ 1\right\} }\big\|\Delta _{j}f_{i}\big\|%
_{t_{i}}\prod_{i\in I_{2}\backslash \left\{ 1\right\} }\big\|Q_{j}f_{i}\big\|%
_{\infty } \\
\leq & c\text{ }2^{j(\theta -\tfrac{n}{p}+\tfrac{n}{p_{1}})}\big\|f_{1}\mid
B_{p_{1},\infty }^{s_{1}}\big\|\prod_{i=2}^{m}\big\|f_{i}\mid
B_{p_{i},\infty }^{n/p_{i}}\big\|,\quad j\in \mathbb{N}_{0}.
\end{align*}%
The result follows from $\mathrm{\eqref{Cond5.7}}$, $\mathrm{\eqref{Cond5.9}}
$ and the fact that $\big\{2^{j(\theta -\tfrac{n}{p}+\tfrac{n}{p_{1}})}\big\}%
\in \ell _{\infty }$.

$\bullet $ \textbf{Substep 2.2.} $s_{i}>\tfrac{n}{p_{i}},i\in I_{3}\subset I$%
\textbf{.} If $s_{i}>\tfrac{n}{p_{i}},i\in I_{3}\subsetneq I$, we can cover
this case by the same arguments given in Step 1. Let $f_{i}\in
B_{p_{i},\infty }^{s_{i}},i\in \left\{ 2,...,m\right\} $. If $s_{i}>\tfrac{n%
}{p_{i}}$, $i\in I\backslash \left\{ 1\right\} $, then we can prove that $%
\Pi _{2}(\cdot ,f_{2},...,f_{m}):F_{p_{1},q}^{s_{1}}\longrightarrow
F_{t,q}^{s_{1}}$. We get $\Pi _{2}(\cdot
,f_{2},...,f_{m}):F_{p_{1},q}^{s_{1}}\longrightarrow F_{p_{1},q}^{s_{1}}$,
by using the embeddings 
\begin{equation*}
B_{v,\infty }^{s_{1}+\tfrac{n}{p_{i_{0}}}+\theta }\hookrightarrow
F_{v,\infty }^{s_{1}+\tfrac{n}{p_{i_{0}}}}\hookrightarrow F_{p_{1},q}^{s_{1}}
\end{equation*}%
for some $i_{0}\in I_{1}\backslash \left\{ 1\right\} $ and $\tfrac{1}{v}=%
\tfrac{1}{p_{1}}+\tfrac{1}{p_{i_{0}}}$, with 
\begin{equation*}
-s_{1}-\tfrac{n}{p_{i_{0}}}<\theta <s_{i_{0}}-\tfrac{n}{p_{i_{0}}},
\end{equation*}%
which is possible because of $s_{1}+s_{2}>0$. We conclude the desired result
using the interpolation method given in Subcase 1.2 in the proof of Theorem %
\ref{multi-positive-smothness}.
\end{proof}

\begin{remark}
We hope that our results will be used in the theory of function spaces,
composition operators and partial differential equations. We refer the
reader to the paper \cite{BD23} were the authors studied the $2$-linear map%
\begin{equation}
A_{p_{1},q}^{s}(\mathbb{R}^{n},|\cdot |^{\alpha })\cdot A_{p_{2},q_{2}}^{r}(%
\mathbb{R}^{n},|\cdot |^{\alpha })\hookrightarrow A_{p,q}^{s}(\mathbb{R}%
^{n},|\cdot |^{\alpha }),  \label{product}
\end{equation}%
induced by 
\begin{equation*}
\left( f_{1},f_{2}\right) \longrightarrow f_{1}\cdot f_{2},
\end{equation*}%
with an application to the continuity of pseudodifferential operators on
Triebel-Lizorkin spaces of power weight. Here $A_{p,q}^{s}(\mathbb{R}%
^{n},|\cdot |^{\alpha })$ stands for either the Besov space $B_{p,q}^{s}(%
\mathbb{R}^{n},|\cdot |^{\alpha })$ or the Triebel-Lizorkin space $%
F_{p,q}^{s}(\mathbb{R}^{n},|\cdot |^{\alpha })$. Also, they have proved the
necessity of the majority assumptions on the parameters.
\end{remark}

\textbf{Acknowledgement}. This work was supported by the General Directorate
of Scientific Research and Technological Development of Algeria and the
General Direction of Higher Education and Training (Grant no.
C00L03UN280120220004), Algeria.

\end{document}